\documentclass[1 [leqno,11pt]{amsart}
\usepackage{amssymb, amsmath}
\def\refer#1{~\ref{#1}}
\def\refeq#1{~(\ref{#1})}
\def\ccite#1{~\cite{#1}}

\def\inte#1{
\displaystyle\mathop{#1\kern0pt}^\circ }

\let\pa=\partial
\let\p=\partial
\let\al=\alpha

\let\d=\delta

\let\r=\rho
\let\s=\sigma
\let\f=\frac
\let\th=\theta

\let\D=\Delta

\let\wt=\widetilde
\let\wh=\widehat

\def\cB{{\mathcal B}}
\def\cC{{\mathcal C}}

\def\cF{{\mathcal F}}

\def\cH{{\mathcal H}}

\def\cS{{\mathcal S}}

\def\cV{{\mathcal V}}

\def\grad{\nabla}

\def\dH{\dot{H}}
\def\dB{\dot{B}}

\def\h{\frak h}

\def\virgp{\raise 2pt\hbox{,}}
\def\cdotpv{\raise 2pt\hbox{;}}

\def\eqdefa{\buildrel\hbox{\footnotesize def}\over =}
\def\Id{\mathop{\rm Id}\nolimits}

\def\C{\mathop{\mathbb C\kern 0pt}\nolimits}
\def\DD{\mathop{\mathbb D\kern 0pt}\nolimits}
\def\EE{\mathop{  {\mathbb E \kern 0pt}}\nolimits}
\def\K{\mathop{\mathbb K\kern 0pt}\nolimits}
\def\N{\mathop{\mathbb N\kern 0pt}\nolimits}
\def\Q{\mathop{\mathbb Q\kern 0pt}\nolimits}
\def\R{\mathop{\mathbb R\kern 0pt}\nolimits}
\def\SS{\mathop{\mathbb S\kern 0pt}\nolimits}
\def\ZZ{\mathop{\mathbb Z\kern 0pt}\nolimits}
\def\TT{\mathop{\mathbb T\kern 0pt}\nolimits}
\newcommand{\ds}{\displaystyle}

\newcommand{\Z}{{\ZZ}}

\newcommand{\Rmnum}[1]{\uppercase\expandafter{\romannumeral #1} }
 \numberwithin{equation}{section}

\def\dive{\mathop{\rm div}\nolimits}
\def\curl{\mathop{\rm curl}\nolimits}

\def\Supp{\mathop{\rm Supp}\nolimits\ }

\def\no{\noindent}
\def\na{\nabla}

\def\vhcurl{v^{\rm h}_{\rm curl}}
\def\vcurl{v^{\rm h}_{\rm curl}}
\def\nablah{\nabla_{\rm h}}
\def\vhdiv{v^{\rm h}_{\rm div}}
\def\vdiv{v^{\rm h}_{\rm div}}

\def\p3{\partial_3}

\def\ph{\partial_{\rm h}}
\def\divh{{\rm div}_{\rm h}}
\def\Laplacianh{\Delta_{\rm h}^{-1}}

\def\u3{u^3}

\def\b3{b^3}
\def\vh{v^{\rm h}}
\def\vl{v^{\ell}}
\def\h{{\rm h}}
\def\v{{\rm v}}

\def\Rh{R^{\rm h}}
\def\Rv{R^{\rm v}}
\def\Th{T^{\rm h}}
\def\Thb{\overline{T}^{\rm h}}
\def\Tv{T^{\rm v}}
\def\Tvb{\overline{T}^{\rm v}}

\def\Lh{L_{\rm h}}
\def\Lv{L_{\rm v}}

\def\xih{\xi_{\rm h}}

\def\om3{\omega^3}

\def\omr{\omega_{\frac r2}}
\def\omro{\omega_{r-1}}

\def\htr{\cH^{\theta,r}}
\def\htrps{\cH^{\theta,r}_{p,s_2}}

\def\dhk{\Delta_k^{\rm h}}
\def\dhkp{\Delta_{k'}^{\rm h}}

\def\Shk{S_{k-1}^{\rm h}}

\def\Svl{S_{\ell-1}^{\rm v}}

\def\Shkp{S_{k'-1}^{\rm h}}

\def\Svlp{S_{\ell'-1}^{\rm v}}

\def\dvl{\Delta_\ell^{\rm v}}
\def\dvlp{\Delta_{\ell'}^{\rm v}}

\def\dj{\Delta_j}
\def\dtj{\widetilde{\Delta}_j}

\def\th{\theta}

\def\Bs{\bigl(\dB^{0}_{s_1,\infty}\bigr)_{\rm{h}}\bigl(\dB^{\frac1p+3\al(r)}_{s_2,\infty}\bigr)_{\rm{v}}}
\def\Bsfg{\bigl(\dB^{-6\al(r)+2\th}_{s_1',1}\bigr)_{\rm{h}}
\bigl(\dB^{\frac1{p'}-3\al(r)-2\th}_{s_2',1}\bigr)_{\rm{v}}}

\newcommand{\beq}{\begin{equation}}
\newcommand{\eeq}{\end{equation}}
\newcommand{\ben}{\begin{eqnarray}}
\newcommand{\een}{\end{eqnarray}}
\newcommand{\beno}{\begin{eqnarray*}}
\newcommand{\eeno}{\end{eqnarray*}}
\newcommand{\andf}{\quad\hbox{and}\quad}
\newcommand{\with}{\quad\hbox{with}\quad}
\newtheorem{defi}{Definition}[section]
\newtheorem{thm}{Theorem}[section]
\newtheorem{lem}{Lemma}[section]
\newtheorem{rmk}{Remark}[section]

\newtheorem{prop}{Proposition}[section]
\renewcommand{\theequation}{\thesection.\arabic{equation}}
\begin{document}
\title[]
{Critical one component anisotropic regularity for 3-D
Navier-Stokes system}

\author[Y. Liu]{Yanlin Liu}
\address[Y. Liu]{department of mathematical sciences, university of science and technology of china, hefei 230026, china,
and Academy of Mathematics $\&$ Systems Science, The Chinese Academy of
Sciences, Beijing 100190, CHINA.} \email{liuyanlin3.14@126.com}
\author[P. Zhang]{Ping Zhang} \address[P. Zhang]{Academy of Mathematics $\&$ Systems Science
and  Hua Loo-Keng Key Laboratory of Mathematics, The Chinese Academy of
Sciences, Beijing 100190, CHINA, and School of Mathematical Sciences, University of Chinese Academy of Sciences, Beijing 100049, China.} \email{zp@amss.ac.cn}

\date{\today}

\begin{abstract}  Let us consider an initial data~$v_0$  for the classical
 3D Navier-Stokes equation with vorticity  belonging to~$L^{\frac 32}\cap L^2$.
  We prove that if the solution associated with~$v_0$ blows up at a finite time~$T^\star$,
  then for any $p\in]4,\infty[,~q_1\in[1,2[,~\mu>0,
~q_2\in\bigl[2,\bigl(1/p+\mu\bigr)^{-1}\bigr[,~\kappa\in ]1,\infty[$,
and any unit vector $e$,
   the~$L^p$ estimate in time of $\bigl\|(v(t)|e)_{\R^3}\bigr\|_{L^{\frac{3p}{p-2}}}^p
+\bigl\|(v(t)|e)_{\R^3}\bigr\|^p_{
\bigl(\dB^{\mu+\f2p+\f2{q_1}-1}_{q_1,\kappa}\bigr)_\h
\bigl(\dB^{\f1{q_2}-\mu}_{q_2,\kappa}\bigr)_\v}$ blows up at~$T^\star$.
\end{abstract}

\maketitle


\setcounter{equation}{0}
\section{Introduction}
In this paper,  we investigate  necessary  conditions for  breakdown of regularity of regular solutions to
the following 3-D  incompressible Navier-Stokes system
\begin{equation*}
(NS)\qquad \left\{\begin{array}{l}
\displaystyle \pa_t v + \dive (v\otimes v) -\D v+\grad p=0, \qquad (t,x)\in\R^+\times\R^3, \\
\displaystyle \dive\, v = 0, \\
\displaystyle  v|_{t=0}=v_0,
\end{array}\right. \label{1.1}
\end{equation*}
where $v=(v^1,v^2, v^3)$ stands for the  fluid  velocity and
$p$ for the scalar pressure function, which guarantees the divergence free condition of the velocity field.

 Let us sum up the fact about this theory introduced in\ccite{fujitakato}
 that will be relevant in our work.
\begin{thm}
\label{fujitakato+}
{\sl Let~$v_0$ be in the homogenneous Sobolev
space~$\dot H^{\frac 12}$. Then there exists a unique  maximal
solution~$v$ in the space~$ C([0,T^\ast[;\dH ^{\frac 1 2   })\cap
L^2_{\rm loc}([0,T^\star[;\dH ^{\frac 32})$. Moreover, if~$T^\star$ is finite,
then for any~$p$ in~$[2,\infty[$, there holds
\beq
\label{blowupbasic}
\int_0^{T^\star}\|v(t,\cdot)\|_{\dH^{\frac 1 2   +\frac 2 p}}^p\,dt=\infty.
\eeq }
\end{thm}
The endpoint   case when~$p= \infty$ in the above theorem  was proved by Kenig and Koch in \cite{KK11},  that is
 if the lifespan $T^\star$ is finite, then~$\ds \limsup_{t\rightarrow T^\star}
\|v(t)\|_{\dot H^{\frac 12}}$ is infinite. This end point case can also be viewed as a consequence of the work\ccite{ISS}  of Escauriaza,  Seregin and {S}ver\'{a}k,
where the authors proved that  if the lifespan $T^\star$ is finite, then~$\ds \limsup_{t\rightarrow T^\star}
\|v(t)\|_{L^3}$ is infinite.

Before preceding, we recall
the following family of spaces from \cite{CZZ}
\begin{defi}
\label{definVorticiyspaces} {\sl For~$r$ in~$\bigl]\frac 32,
2\bigr]$, we denote by~$\cV^r$ the space of divergence free vector
fields with the vorticity of which belongs to~$L^{\frac 32}\cap
L^r$. }
\end{defi}

In ~\cite{CZ5}, Chemin and the second author proved  the following component-reduction version
of Theorem \ref{fujitakato+}:

\begin{thm}
\label{thmCZ5} {\sl Let $v$ be the  unique maximal solution~$v$ of $(NS)$
associated with~$v_0\in \cV^{\frac32}.$
If its lifespan ~$T^\star$ is finite,  then for any~$p\in]4,6[$ and any unit vector~$e$, we have
\begin{equation}\label{blowupCZ5}
\int_0^{T^\star}\|(v(t)|e)_{\R^3}\|_{\dH^{\frac 1 2   +\frac 2 p }}^p\,dt=\infty.
\end{equation}
}
\end{thm}
Recently in \cite{CZZ}, the blow-up criteria \eqref{blowupCZ5} was improved
to  $p\in]4,\infty[$ provided that the initial data $v_0$ belongs to a more regular space $\cV^{\f32}\cap \cV^{2}$. While the authors in \cite{LeiZhao} dealt with the remaining case $p\in[2,4]$.
One may check \cite{CZ5} for more references concerning the regularity criteria of solutions to three-dimensional Navier-Stokes system.

Another improvement of Theorem \ref{fujitakato+} is the well-known
Ladyzhenskaya-Prodi-Serrin criteria, more precisely, if the life-span $T^\ast$ for smooth enough solutions of $(NS)$ is finite, then
\beq
\label{blowupLPS}
\int_0^{T^\star}\|v(t,\cdot)\|_{L^q}^p\,dt=\infty,\quad\mbox{where}\quad \frac2p+\frac3q=1.
\eeq
We point out that, in \eqref{blowupLPS}, there is no requirement on the derivative estimate of the solution $v.$ Whereas
 the one component version \eqref{blowupCZ5} requires  more than half derivative  estimate on $v.$
The purpose of this paper is to reduce the order of derivative estimate in \eqref{blowupCZ5}.

Let us mention that,
as in\ccite{CDGG}, \ccite{CZ1}, \ccite{CZ5} and \ccite{Pa02}, the
definitions of the function spaces we are going to work with require
anisotropic dyadic decomposition of the Fourier variables.
Let us first recall some basic facts on anisotropic Littlewood-Paley theory from \cite{BCD}
\begin{equation}\begin{split}\label{defparaproduct}
&\Delta_ja=\cF^{-1}(\varphi(2^{-j}|\xi|)\widehat{a}),
 \quad \Delta_k^{\rm h}a=\cF^{-1}(\varphi(2^{-k}|\xi_{\rm h}|)\widehat{a}),
 \quad \Delta_\ell^{\rm v}a =\cF^{-1}(\varphi(2^{-\ell}|\xi_3|)\widehat{a}),\\
&S_ja=\cF^{-1}(\chi(2^{-j}|\xi|)\widehat{a}),
\quad S^{\rm h}_ka=\cF^{-1}(\chi(2^{-k}|\xi_{\rm h}|)\widehat{a}),
\quad\ S^{\rm v}_\ell a =\cF^{-1}(\chi(2^{-\ell}|\xi_3|)\widehat{a}) \quad \mbox{and}\\
&\dtj=\Delta_{j-1}+\Delta_j +\Delta_{j+1},
\quad \widetilde{\Delta}_k^{\rm h}=\Delta_{k-1}^{\rm h}+\Delta_k^{\rm h}+\Delta_{k+1}^{\rm h},
\quad \widetilde{\Delta}_\ell^{\rm v}=\Delta_{\ell-1}^{\rm v}+\Delta_\ell^{\rm v}+\Delta_{\ell+1}^{\rm v},
\end{split}\end{equation}
where $\xi_{\rm h}=(\xi_1,\xi_2),$ $\cF a$ and
$\widehat{a}$ denote the Fourier transform of the distribution $a,$
$\chi(\tau)  $ and~$\varphi(\tau)$ are smooth functions such that
 \beno
&&\Supp \varphi \subset \Bigl\{\tau \in \R\,/\  \ \frac34 \leq
|\tau| \leq \frac83 \Bigr\}\andf \  \ \forall
 \tau>0\,,\ \sum_{j\in\Z}\varphi(2^{-j}\tau)=1,\\
&&\Supp \chi \subset \Bigl\{\tau \in \R\,/\  \ \ |\tau|  \leq
\frac43 \Bigr\}\quad \ \ \ \andf \  \ \forall
 \tau\in\R\,,\  \chi(\tau)+ \sum_{j\geq
0}\varphi(2^{-j}\tau)=1.
 \eeno

 \begin{defi}\label{anibesov}
{\sl Let us define the space $\bigl(\dB^{s_1}_{p_1,r_1}\bigr)_{\rm
h}\bigl(\dB^{s_2}_{p_2,r_2}\bigr)_{\rm v}$ as the space of
distribution in~$\cS'_h$  such that
$$
\|u\|_{\bigl(\dB^{s_1}_{p_1,r_1}\bigr)_{\rm
h}\bigl(\dB^{s_2}_{p_2,r_2}\bigr)_{\rm v}}\eqdefa \biggl(\sum_{k\in\Z}
2^{r_1ks_1} \Bigl(\sum_{\ell\in\Z}2^{r_2\ell s_2}\|\D_k^{\rm
h}\D_\ell^{\rm
v}u\|_{\Lh^{p_1}\Lv^{p_2}}^{r_2}\Bigr)^{{r_1}/{r_2}}\biggr)^{1/{r_1}}
$$
is finite. When $p_1=p_2=p$, $r_1=r_2=r$, we briefly denote $\bigl(\dB^{s_1}_{p,r}\bigr)_{\rm
h}\bigl(\dB^{s_2}_{p,r}\bigr)_{\rm v}$ as $\dB^{s_1,s_2}_{p,r}$.
}
\end{defi}

The main result of this paper states as follows:

\begin{thm}\label{thmmain}
{\sl Let $v$ be the unique maximal solution of $(NS)$
associated with initial data $v_0\in\cV^{2}.$
If its lifespan $T^\star$ is finite, then for any $p\in]4,\infty[,~q_1\in[1,2[,~\mu>0,
~q_2\in\bigl[2,\bigl(1/p+\mu\bigr)^{-1}\bigr[,~\kappa\in ]1,\infty[$,
and any unit vector $e$, there must hold
\begin{equation}\label{blowupmain}
\int_0^{T^\star}\Bigl(\bigl\|(v(t)|e)_{\R^3}\bigr\|_{L^{\frac{3p}{p-2}}}^p
+\bigl\|(v(t)|e)_{\R^3}\bigr\|^p_{
\bigl(\dB^{\mu+\f2p+\f2{q_1}-1}_{q_1,\kappa}\bigr)_\h
\bigl(\dB^{\f1{q_2}-\mu}_{q_2,\kappa}\bigr)_\v}\Bigr)\,dt=\infty.
\end{equation}
}\end{thm}

\begin{rmk}
We mention that the first term in \eqref{blowupmain} corresponds to the one component version of  $L^p_{T^\star}\bigl(L^{\frac{3p}{p-2}}\bigr)$
 Ladyzhenskaya-Prodi-Serrin  criteria.
While  the  second term in \eqref{blowupmain} requires $\mu+\f2p+\f2{q_1}-1$ order derivative estimate of $(v|e)$ for horizontal variables,  which
can be arbitrarily close to zero provided  that we choose $q_1$ sufficiently close to $2$,
$\mu$ small enough and $p,~q_2$ large enough. And similar comment for the vertical derivative estimate of $(v|e)$.
Yet we can not succeed in reducing the derivative estimate
to be zero-th order.
\end{rmk}

At the end of this section, let us give some notations which will be used throughout this paper.
$C$ stands for some real positive constant which may be different in each occurrence.
Sometimes we use the notation $a\lesssim b$ for the inequality $a\leq Cb$.
For a Banach space $B$, we shall use the shorthand $L^p_TB$ for $\bigl\|\|\cdot\|_B\bigr\|_{L^p(0,T;dt)}$.
We always denote $\left(c_{k,\ell}\right)_{k,\ell\in\Z}$ to be a generic
element of $\ell^2(\Z^2)$ so that $\sum_{k,\ell\in\Z}c_{k,\ell}^2=1$,
and $\left(d_{k,\ell}\right)_{k,\ell\in\Z}$ to be a generic
element of $\ell^1(\Z^2)$ so that $\sum_{k,\ell\in\Z}c_{k,\ell}=1$.
And for any index $p\in[1,\infty]$, we shall use $p'$ to denote its conjugate index, i.e.
$1/p+1/p'=1$.

\section{Strategy of the proof of Theorem \ref{thmmain}}

In what follows, we always denote
\beq\label{nal} \al(r)\eqdefa
\frac1r-\f12,\quad  \cB_{q_1,q_2,r}^{\mu,p}\eqdefa\bigl(\dB^{\mu+\f2p+\f2{q_1}-1}_{q_1,\frac{2r}{2-r}}\bigr)_\h
\bigl(\dB^{\f1{q_2}-\mu}_{q_2,\frac{2r}{2-r}}\bigr)_\v \andf a_\al \eqdefa  \frac a {|a|} |a|^\al\ \forall \al\in ]0,1[. \eeq
Without loss of generality, we may
assume that the unit vector $e$ in Theorem \ref{thmmain} is the vertical
vector $e_3\eqdefa (0,0,1)$.
Note that for any given $p\in]4,\infty[,~\mu>0,
~q_2\in\bigl[2,\bigl(\frac1p+\mu\bigr)^{-1}\bigr[,~\kappa\in ]1,\infty[$, we can find some $r<2$
which is arbitrarily close to $2,$ such that
$p\in\bigl]4,\frac {2r} {2-r} \bigr[,~\mu>\al(r),~q_2\in\bigl[2,\bigl(1/p+3\al(r)+\mu\bigr)^{-1}\bigr[,$ and $
~\kappa<\frac{2r}{2-r}$. As a convention in the rest of the paper, we always assume that $r$ satisfies the above assumptions.
Then Theorem \refer{thmmain} is a direct consequence of the following
one:
\begin{thm}
\label{thmmain1}
 {\sl Let  $r\in\bigl[9/5,2\bigr[$,
$p\in\bigl]4,\frac {2r} {2-r} \bigr[,~q_1\in[1,2[,~\mu\in\bigl]\al(r),\frac12-\frac1p-3\al(r)\bigr[$,
$q_2\in\bigl[2,\bigl(1/p+3\al(r)+\mu\bigr)^{-1}\bigr[$, and ~$v_0$ in $\cV^r$. If the
lifespan~$T^\star$ of the unique maximal solution~$v$ of~$(NS)$ is finite,  then
there holds
\beq
\label{k.1}
\int_0^{T^\star}\|v^3\|_{SC}^p\,dt=\infty\with \|v^3\|_{SC}\eqdefa\|v^3\|_{L^{\frac{3p}{p-2}}}
+\|v^3\|_{\cB_{q_1,q_2,r}^{\mu,p}}.
\eeq
}\end{thm}

In the rest of this section, we shall present the strategy of the proof to Theorem \ref{thmmain1}.

Before preceding, we denote $\Omega\eqdefa \curl v$ to be the vorticity of the velocity field,
$\omega=\pa_1 v^2-\pa_2 v^1$ to be the third component of $\Omega$,
and $\nabla_{\rm
h}^\perp=(-\pa_2,\pa_1),~\D_{\rm h}=\pa_1^2+\pa_2^2$.
Due to $\dive v=0$, we deduce the following version of  Biot-Savart's law in the horizontal variables
\beq
\label{a.1wrt} v^{\rm h}=\vcurl+\vdiv,\quad\mbox{with}\quad
\vcurl
\eqdefa\nablah^\perp \D_{\rm h}^{-1} \omega \andf \vdiv\eqdefa
-\nablah  \D_{\rm h}^{-1}
\partial_3v^3.
\eeq
Then we can reformulate the Navier-Stokes equations $(NS)$ in terms of $\omega$ and $\pa_3 v^3$:
$$
(\wt {NS})
\quad\left\{
\begin{array}{c}
\partial_t\omega+v\cdot\nabla\omega -\D\omega= \partial_3v^3\omega +\partial_2v^3\partial_3v^1-\partial_1v^3\partial_3v^2,\\
\partial_t \partial_3v^3 +v\cdot\nabla\partial_3v^3-\D\partial_3v^3+\partial_3v\cdot\nabla v^3
=-\partial_3^2\D^{-1} \Bigl(\ds\sum_{\ell,m=1}^3 \partial_\ell
v^m\partial_mv^\ell\Bigr).
\end{array}
\right.
$$

The first step of the proof of Theorem\refer{thmmain1} is
the following
 proposition:
  \begin{prop}
 \label{prop2.1}
 {\sl Under the assumptions of Theorem \ref{thmmain1} and for any $\th\in]0,\al(r)[$,
 a constant $C$ exists such that, for any $t<T^\star$, we have
\begin{equation}\begin{split}\label{prop1}
\frac 1r \bigl\|\omega_{\frac{r}2}& (t) \bigr\|_{L^2}^2 +
\frac{r-1}{r^2}\int_0^t \bigl\|\nabla
\omega_{\frac{r}2}(t')\bigr\|_{L^2}^2dt'\\
& \leq
\Bigl(\frac1r\|\omega_0\|_{L^r}^r
+C\bigl(\int_0^t\|\pa_3^2 v^3(t')\|_{\htr}^2dt'\bigr)^{\frac{r}2}
\Bigr)\times \exp\Big(C\int_0^t\|v^3(t')\|_{SC}^{p}dt'\Bigr).
\end{split}\end{equation}
}\end{prop}

The proof of Proposition \ref{prop2.1} will be the purpose of Section \ref{Sect3}.

We remark that when $p_1=p_2=r_1=r_2=2,$ the anisotropic Besov space
$\dB^{s_1,s_2}_{2,2},$ given by Definition \ref{anibesov} coincides with the
classical homogeneous anisotropic Sobolev space $\dH^{s_1,s_2}$
defined as follows:
\begin{defi}
\label{defhtr} {\sl For~$(s_1,s_2)$ in~$\R^2$, $\dH^{s_1,s_2} $ denotes
the space of tempered distribution~$a$  such~that
 $$
\|a\|^2_{\dH^{s_1,s_2}} \eqdefa \int_{\R^3} |\xi_{\rm
h}|^{2s_1}|\xi_3|^{2s_2} |\wh a (\xi)|^2d\xi <\infty\with \xi_{\rm
h}=(\xi_1,\xi_2).
 $$
 For $\al(r)$ given by \eqref{nal} and ~$\theta$ in~$]0,\al(r)[$, we denote~$\cH^{\theta,r}\eqdefa \dH^{-3\al(r)+\theta,-\theta}$.
 }
 \end{defi}

\medbreak
Next we are going to deal with the estimate of~$\|\pa^2_3v^3\|_{L^2_t(\cH^{\theta,r})}$.
Due to $\Omega_0\eqdefa \Omega|_{t=0}\in L^r$ and $\dive v_0=0,$ we have
\begin{equation}\begin{split}
\|\pa_3 v^3_0\|_{\cH^{\theta,r}}^2=&\int_{|\xi_3|\leq|\xi_{\rm h}|}|\xi_{\rm h}|^{-6\alpha(r)+2\theta}|\xi_3|^{-2\theta}
\bigl|\cF(\pa_3\v^3_0)(\xi)\bigr|^2 d\xi\\
&+\int_{|\xi_{\rm h}|\leq|\xi_3|}|\xi_{\rm h}|^{-6\alpha(r)+2\theta}|\xi_3|^{-2\theta}
\bigl|\cF\bigl(-\divh v_0^{\rm h}\bigr)(\xi)\bigr|^2 d\xi\\
\leq& \|v_0^3\|^2_{H^{1-3\alpha(r)}}\lesssim\|\Omega_0\|^2_{L^r}.
\end{split}\end{equation}
By  performing  $\cH^{\theta,r}$-norm energy estimate
to the $\pa_3 v^3$ equation of $(\wt {NS})$, we shall prove the following proposition in  Section \ref{Sect4}:
\begin{prop}
 \label{prop2.2}
 {\sl Under the assumption of Proposition \ref{prop2.1}, for any~$t<T^\star,$
we have
\begin{equation}\begin{split}\label{b.4bqp}
& \|\pa_3v^3(t)\|_{\cH^{\theta,r}}^2
 +\int_0^t\|\na\pa_3v^3(t')\|_{\cH^{\theta,r}}^2\,dt' \leq C\exp \Bigl(C\int_0^t
\|v^3(t')\|^{p}_{SC}\,dt' \Bigr)\\
 &\qquad\qquad{}\times\biggl(\|\Omega_0\|_{L^{r}}^2
 + \int_0^t \Bigl(\|v^3(t')\|_{SC}\bigl\|\omega_{\f{r}2}(t')\bigr\|_{L^2}^{2\left(2\al(r)+\frac 1 p \right)}
\bigl\|\na\omega_{\f{r}2}(t')\bigr\|_{L^2}^{2\bigl(1-\frac 1p\bigr)}\\
&\qquad\qquad\qquad\qquad\qquad{}+\|v^3(t')\|_{SC}^2\bigl\|\omega_{\f{r}2}(t')\bigr\|_{L^2}^{4\left(\al(r)+\frac 1 p \right)}
\bigl\|\na\omega_{\f{r}2}(t')\bigr\|_{L^2}^{2\left(1-\frac 2 p \right)}\Bigr)\,dt'\biggr).
\end{split}\end{equation}
}\end{prop}

With Propositions \ref{prop2.1} and \ref{prop2.2}, we can  repeat  the arguments in Section 6 of \cite{CZZ}
to complete the proof of Theorem \ref{thmmain1}. For completeness, we shall sketch the proof below.

\begin{prop}\label{prop6.1}
{\sl Under the assumption of Proposition \ref{prop2.1}, for any~$t<T^\star,$
we have
\ben
\label{6.1}
\bigl\|\omr(t)\bigr\|_{L^2}^{2(1+2p\alpha(r))}+\bigl\|\nabla\omr\bigr\|_{L_t^2(L^2)}^{2(1+2p\alpha(r))}
&\leq& C\|\Omega_0\|_{L^r}^{r(1+2p\alpha(r))}\cdot \varepsilon(t),\\
\label{6.2}
\bigl\|\p3 v^3(t)\bigr\|_{\htr}^{2}+\bigl\|\nabla\p3 v^3\bigr\|_{L_t^2(\htr)}^{2}
&\leq& C\|\Omega_0\|_{L^r}^{2}\cdot \varepsilon(t) \with\\
\varepsilon(t)&\eqdefa&\exp\Bigl(C\exp\bigl(C\int_0^t\bigl\|v^3(t')\bigr\|_{SC}^p dt'\bigr)\Bigr). \nonumber
\een
}\end{prop}
\begin{proof} For $p\in [1,\infty]$ and any $t\in[0,T^\star[$, let us denote
\begin{equation}
p'=\f{p}{p-1} \andf e(t)\eqdefa C\exp\bigl(C\int_0^t\bigl\|v^3(t')\bigr\|_{SC}^p dt'\bigr),
\end{equation}
where the constant $C$ may change from line to line. Then it follows from Proposition \ref{prop2.2}  that
\begin{equation}\label{6.4}
\Bigl(\int_0^t\|\na\pa_3v^3(t')\|_{\cH^{\theta,r}}^2\,dt'\Bigr)^{\frac r2} e(t) \lesssim e(t)
\bigl(\|\Omega_0\|_{L^{r}}^r+\Rmnum{2}_1(t)+\Rmnum{2}_2(t)\bigr)
\end{equation}
where
\begin{align*}
&\Rmnum{2}_1(t)\eqdefa\Bigl( \int_0^t \|v^3(t')\|_{SC}\bigl\|\omega_{\f{r}2}(t')\bigr\|_{L^2}^{2\left(2\al(r)+\frac 1 p \right)}
\bigl\|\na\omega_{\f{r}2}(t')\bigr\|_{L^2}^{\frac 2 {p'}}\,dt'\Bigr)^{\frac r2}\quad\mbox{and}\\
&\Rmnum{2}_2(t)\eqdefa \Bigl( \int_0^t \bigl\|v^3(t')\bigr\|_{SC}^2\bigl\|\omega_{\f{r}2}(t')\bigr\|_{L^2}^{4\left(\al(r)+\frac 1 p \right)}
\bigl\|\na\omega_{\f{r}2}(t')\bigr\|_{L^2}^{2\left(1-\frac 2 p \right)}\,dt'\Bigr)^{\frac r2}.
\end{align*}
Applying H\"{o}lder's and Young's inequalities yields
\begin{equation}\begin{split}\label{6.5}
e(t)&\Rmnum{2}_1(t)\leq e(t)
\Bigl(\int_0^t\|v^3(t')\|_{SC}^p\bigl\|\omega_{\f{r}2}(t')\bigr\|_{L^2}^{2\left(1+2p\al(r)\right)}dt'\Bigr)^{\frac r2\cdot\frac1p}
\Bigl( \int_0^t\bigl\|\na\omega_{\f{r}2}(t')\bigr\|_{L^2}^2 dt'\Bigr)^{\frac r2\cdot\frac{1}{p'}}\\
&\leq \frac{r-1}{3r^2}\int_0^t\bigl\|\na\omega_{\f{r}2}(t')\bigr\|_{L^2}^2 dt'
+e(t)\Bigl(\int_0^t\|v^3(t')\|_{SC}^p\bigl\|\omega_{\f{r}2}(t')\bigr\|_{L^2}^{2\left(1+2p\al(r)\right)}dt'\Bigr)^{\frac{1}{1+2p\alpha(r)}},
\end{split}\end{equation}
and
\begin{equation*}\begin{split}\label{6.6}
e(t)&\Rmnum{2}_2(t)\leq e(t)
\Bigl(\int_0^t\|v^3(t')\|_{SC}^p\bigl\|\omega_{\f{r}2}(t')\bigr\|_{L^2}^{2\left(1+p\al(r)\right)}dt'\Bigr)^{\frac r2\cdot\frac2p}
\Bigl( \int_0^t\bigl\|\na\omega_{\f{r}2}(t')\bigr\|_{L^2}^2 dt'\Bigr)^{\frac r2\cdot\bigl(1-\frac{2}{p}\bigr)}\\
&\leq \frac{r-1}{3r^2}\int_0^t\bigl\|\na\omega_{\f{r}2}(t')\bigr\|_{L^2}^2 dt'
+e(t)\Bigl(\int_0^t\|v^3(t')\|_{SC}^p\bigl\|\omega_{\f{r}2}(t')\bigr\|_{L^2}^{2\left(1+p\al(r)\right)}dt'\Bigr)^{\frac{1}{1+p\alpha(r)}}.
\end{split}\end{equation*}
It is easy to observe that
\begin{align*}
\Bigl(\int_0^t\|v^3(t')\|_{SC}^p\bigl\|\omega_{\f{r}2}(t')\bigr\|_{L^2}^{2\left(1+p\al(r)\right)}&dt'\Bigr)^{\frac{1}{1+p\alpha(r)}}
\leq \Bigl(\int_0^t\|v^3(t')\|_{SC}^p dt'\Bigr)^{\frac{p\al(r)}{(1+p\alpha(r))(1+2p\alpha(r))}}\\
&\times \Bigl(\int_0^t\|v^3(t')\|_{SC}^p\bigl\|\omega_{\f{r}2}(t')\bigr\|_{L^2}^{2\left(1+2p\al(r)\right)}dt'\Bigr)^{\frac{1}{1+2p\alpha(r)}}.
\end{align*}
Thus we achieve
$$e(t)\Rmnum{2}_2(t)\leq \frac{r-1}{3r^2}\int_0^t\bigl\|\na\omega_{\f{r}2}(t')\bigr\|_{L^2}^2 dt'
+e(t)\Bigl(\int_0^t\|v^3(t')\|_{SC}^p\bigl\|\omega_{\f{r}2}(t')\bigr\|_{L^2}^{2\left(1+2p\al(r)\right)}dt'\Bigr)^{\frac{1}{1+2p\alpha(r)}}.$$
Inserting the above inequality and \eqref{6.5} into \eqref{6.4} gives, for any $t$ in $[0,T^\star[$
\begin{equation}\begin{split}\label{6.7}
\Bigl(\int_0^t\|\na\pa_3v^3(t')\|_{\cH^{\theta,r}}^2\,dt'\Bigr)^{\frac r2} e(t)
&\leq e(t)\|\Omega_0\|_{L^{r}}^r+\frac{2(r-1)}{3r^2}\int_0^t\bigl\|\na\omega_{\f{r}2}(t')\bigr\|_{L^2}^2 dt'\\
&+e(t)\Bigl(\int_0^t\|v^3(t')\|_{SC}^p\bigl\|\omega_{\f{r}2}(t')\bigr\|_{L^2}^{2\left(1+2p\al(r)\right)}dt'\Bigr)^{\frac{1}{1+2p\alpha(r)}}.
\end{split}\end{equation}
Substituting \eqref{6.7} into \eqref{prop1}, we infer that
\begin{equation}\begin{split}\label{6.8}
\frac 1r \bigl\|\omega_{\frac{r}2}(t) \bigr\|_{L^2}^2 +
\frac{r-1}{3r^2}&\int_0^t \bigl\|\nabla
\omega_{\frac{r}2}(t')\bigr\|_{L^2}^2dt' \leq e(t)\|\Omega_0\|_{L^{r}}^r\\
&+e(t)\Bigl(\int_0^t\|v^3(t')\|_{SC}^p\bigl\|\omega_{\f{r}2}(t')\bigr\|_{L^2}^{2\left(1+2p\al(r)\right)}dt'\Bigr)^{\frac{1}{1+2p\alpha(r)}}.
\end{split}\end{equation}
Taking the power $1+2p\alpha(r)$ to the above  inequality, and then applying Gronwall's lemma leads to Inequality \eqref{6.1}.

On the other hand, it follows from Proposition \ref{prop2.2} that, for any $t\in[0,T^\star[$,
\begin{equation}\begin{split}\label{6.9}
& \|\pa_3v^3(t)\|_{\cH^{\theta,r}}^2
 +\int_0^t\|\na\pa_3v^3(t')\|_{\cH^{\theta,r}}^2\,dt' \\
 &\qquad\qquad{}\leq e(t)\biggl(\|\Omega_0\|_{L^{r}}^2
 + \bigl\|v^3\bigr\|_{L^p_t(SC)}\bigl\|\omega_{\f{r}2}\bigr\|_{L^\infty_t(L^2)}^{2\left(2\al(r)+\frac 1 p \right)}
\bigl\|\na\omega_{\f{r}2}\bigr\|_{L^2_t(L^2)}^{\frac 2 {p'}}\\
&\qquad\qquad\qquad\qquad{}+\|v^3\|_{L^p_t(SC)}^2\bigl\|\omega_{\f{r}2}\bigr\|_{L^\infty_t(L^2)}^{4\left(\al(r)+\frac 1 p \right)}
\bigl\|\na\omega_{\f{r}2}\bigr\|_{L^2_t(L^2)}^{2\left(1-\frac 2 p \right)}\biggr).
\end{split}\end{equation}
Inserting the Estimate \eqref{6.1} into  \eqref{6.9}  concludes the proof of \eqref{6.2} and thus Proposition \ref{prop6.1}.
\end{proof}

Before proceeding, let us recall the following regularity criteria from \cite{CZ5}:

\begin{thm}[Theorem 2.1 of \cite{CZ5}]\label{thm6.1}
{\sl Let $v$ be a solution of $(NS)$ in the space $C([0,T^\star[;\dH^{\frac12})\bigcap$
$L^2_{loc}([0,T^\star[;\dH^{\frac32})$. If $T^\star$ is the maximal time of existence and $T^\star<\infty$,
then for the norm $\mathcal{B}_p\eqdefa \dB^{-2+\frac2p}_{\infty,\infty}$, and any $(p_{k,\ell})$ in $]1,\infty[\, ^9$, we have
$$\sum_{1\leq k,\ell \leq 3} \int_0^{T^\star}\bigl\|\pa_\ell v^k(t)\bigr\|_{\mathcal{B}_{p_{k,\ell}}}^{p_{k,\ell}} dt=\infty.$$
}\end{thm}

Now we are in a position  to complete the  proof of  Theorem \ref{thmmain1}.

\begin{proof}[Proof of Theorem \ref{thmmain1}] If we assume that $\int_0^{T^\star}\bigl\|v^3(t)\bigr\|_{SC}^p dt$ is finite, we deduce from Proposition \ref{prop6.1}
 that  the following quantities
\begin{equation}\label{6.10}
\|\omega\|_{L^\infty([0,T^\star[;L^r)},\quad \int_0^{T^\star}\bigl\|\nabla\omr(t)\bigr\|_{L^2}^2 dt,
\quad \mbox{and} \quad \int_0^{T^\star}\bigl\|\nabla\p3 v^3(t)\bigr\|_{\htr}^2 dt
\end{equation} are finite.

It follows from Lemma \ref{lemBern} that
$$\max\limits_{1\leq \ell \leq 3}\bigl\|\pa_\ell v^3\bigr\|_{\cB_p}
\lesssim\sup\limits_{j\in\Z}2^{j\bigl(-1+\frac2p\bigr)}\bigl\|\dj v^3\bigr\|_{L^\infty}
\lesssim \sup\limits_{j\in\Z}\bigl\|\dj v^3\bigr\|_{L^{\frac{3p}{p-2}}}
\lesssim \bigl\|v^3\bigr\|_{L^{\frac{3p}{p-2}}},$$
which together with \eqref{k.1} ensures that
\begin{equation}\label{thm6.1.1}
\max\limits_{1\leq \ell \leq 3}\int_0^{T^\star} \bigl\|\pa_\ell v^3(t)\bigr\|_{\cB_p}^p dt
\lesssim \int_0^{T^\star}\bigl\|v^3(t)\bigr\|_{SC}^p dt<\infty.
\end{equation}

Let us turn the estimate of the horizontal components of the velocity field. For $\vhdiv=-\nablah\Laplacianh\p3 v^3$, we get,
by using  \eqref{thm6.1.1}, that
\begin{equation}\label{thm6.1.2}
\int_0^{T^\star}\bigl\|\nablah\vhdiv(t)\bigr\|_{\cB_p}^p dt
\lesssim \int_0^{T^\star}\bigl\|\p3v^3(t)\bigr\|_{\cB_p}^p dt<\infty.
\end{equation}
While for any distribution $a$, we deduce from Lemma \ref{lemBern} that
\begin{equation}\begin{split}\label{thm6.1.3}
\|\dj a\|_{L^\infty}& \lesssim\sum_{k\leq j+1}\sum_{\ell\leq j+1} 2^k 2^{\frac{\ell}{2}}\|\dhk\dvl a\|_{L^2}\\
& \lesssim \|a\|_{\dH^{1-3\alpha(r)+\theta,-\theta}} \sum_{k\leq j+1}\sum_{\ell\leq j+1}
2^{k(3\alpha(r)-\theta)}2^{\ell\bigl(\frac12+\theta\bigr)}\\
& \lesssim 2^{j\bigl(2-\frac{3}{r'}\bigr)} \|a\|_{\dH^{1-3\alpha(r)+\theta,-\theta}}.
\end{split}\end{equation}
Let $q(r)\eqdefa \frac{2r'}{3}$, \eqref{thm6.1.3} implies
\begin{equation}\label{thm6.1.4}
\bigl\|\p3\vhdiv\bigr\|_{\cB_{\frac{2r'}{3}}}\lesssim\bigl\|\p3\vhdiv\bigr\|_{\dH^{1-3\alpha(r)+\theta,-\theta}}
=\bigl\|\nablah\Laplacianh\p3^2 v^3\bigr\|_{\dH^{1-3\alpha(r)+\theta,-\theta}}\lesssim\bigl\|\p3^2 v^3\bigr\|_{\htr}.
\end{equation}
Due to $r\in]\frac32,2[$, $q(r)\in]\frac43,2[$  then we get, by  applying \eqref{thm6.1.4} and H\"{o}lder inequality, that
\begin{equation}\label{thm6.1.5}
\int_0^{T^\star}\bigl\|\p3\vhdiv(t)\bigr\|_{\cB_{q(r)}}^{q(r)} dt\lesssim T^{\star\bigl(1-\frac{q(r)}{2}\bigr)}
\Bigl(\int_0^{T^\star}\bigl\|\p3^2 v^3(t)\bigr\|_{\htr}^2 dt\Bigr)^{\frac{q(r)}{2}}<\infty.
\end{equation}

On the other hand, we deduce from Lemma \ref{lemBern} that
\begin{equation}\label{thm6.1.6}
\bigl\|\nablah\vhcurl\bigr\|_{\cB_{q(r)}}\lesssim\bigl\|\ph^2\Laplacianh\omega\bigr\|_{\dH^{1-3\alpha(r)}}
\lesssim\bigl\|\nabla\omega\bigr\|_{\dH^{-3\alpha(r)}}\lesssim\bigl\|\nabla\omega\bigr\|_{L^r}.
\end{equation}
Applying Lemma \ref{lemBern} once again and  using the fact that $r<2$, we infer
\begin{align*}
2^{j\bigl(-2+\frac{3}{r'}\bigr)}\bigl\|\dj\p3\vhcurl\bigr\|_{L^\infty}
&\lesssim 2^{j\bigl(-2+\frac{3}{r'}\bigr)}\sum_{k\leq j+1}\sum_{\ell\leq j+1}
\|\dhk\dvl\p3\nablah^\perp\Laplacianh\omega\|_{L^\infty}\\
&\lesssim \|\p3\omega\|_{L^r} 2^{j\bigl(-2+\frac{3}{r'}\bigr)}\sum_{k\leq j+1}\sum_{\ell\leq j+1}
2^{k\bigl(\frac2r-1\bigr)}2^{\frac{\ell}{r}}\lesssim \|\p3\omega\|_{L^r}.
\end{align*}
This together with the Estimate (\ref{thm6.1.6}) and  Lemma \ref{BiotSavartomega} ensures that
$$\bigl\|\nabla\vhcurl(t)\bigr\|_{\cB_{q(r)}}\lesssim\|\nabla\omega(t)\|_{L^r}
\lesssim\bigl\|\omr\bigr\|_{L^\infty([0,T^\star[;L^2)}^{\frac2r-1}\bigl\|\nabla\omr(t)\bigr\|_{L^2}.$$
Using again the fact that $q(r)\in]\frac43,2[$, we get, by using the H\"{o}lder inequality, that
\begin{equation}\label{thm6.1.7}
\int_0^{T^\star}\bigl\|\nabla\vhcurl(t)\bigr\|_{\cB_{q(r)}}^{q(r)} dt\lesssim T^{\star\bigl(1-\frac{q(r)}{2}\bigr)}
\bigl\|\omr\bigr\|_{L^\infty([0,T^\star[;L^2)}^{\frac2r-1}
\Bigl(\int_0^{T^\star}\bigl\|\nabla \omr(t)\bigr\|_{L^2}^2 dt\Bigr)^{\frac{q(r)}{2}}<\infty.
\end{equation}

With the estimates \eqref{thm6.1.1},~\eqref{thm6.1.2},~\eqref{thm6.1.5} and \eqref{thm6.1.7},  Theorem \ref{thmmain1} is a direct
consequence of  Theorem \ref{thm6.1}.
\end{proof}

Finally in the Appendix \ref{apB}, we shall collect some basic facts on Littlwood-Paley theory from \cite{BCD} and some technical lemmas from
\cite{CZ5, CZZ}. While in Appendix \ref{apA}, we present some technical details which will be used in the proof of Proposition \ref{S4prop1}.

\setcounter{equation}{0}
\section{Proof of the  estimate for the horizontal vorticity}\label{Sect3}

 The purpose of this section to present the proof of Proposition \ref{prop2.1}.
 Let us first recall  the $\omega$  equation of ~$(\wt {NS})$ that
 $$
 \partial_t\omega+v\cdot\nabla\omega -\Delta\omega=   \partial_3v^3\omega +\partial_2v^3\partial_3v^1-\partial_1v^3\partial_3v^2.
 $$
By applying Lemma 3.1
 of \cite{CZ5}, we obtain
 \beq
\label{estimateprop1}
\begin{split}
&\frac 1r \bigl\|\omega_{\frac{r}2}(t) \bigr\|_{L^2}^2  +
\frac{4(r-1)}{r^2}\int_0^t \bigl\|\nabla
\omega_{\frac{r}2}(t')\bigr\|_{L^2}^2\,dt' = \frac1r
\bigl\||\omega_0|^{\frac{r}2}
\bigr\|_{L^2}^2 +\sum_{\ell=1}^3 F_\ell(t) \with\\
&F_1(t) \eqdefa  \int_0^t \int_{\R^3} \pa_3v^3 |\omega|^{r}\,dx\,dt'\,,\\
&F_2(t) \eqdefa  \int_0^t \int_{\R^3}\bigl(\pa_2v^3\pa_3v_{\rm curl}^1-\pa_1v^3\pa_3v_{\rm curl}^2\bigr) \omega_{r-1}\,dx\,dt'\andf\\
&F_3(t) \eqdefa  \int_0^t \int_{\R^3}\bigl(\pa_2v^3\pa_3v_{\rm
div}^1-\pa_1v^3\pa_3v_{\rm div}^2 \bigr) \omega_{r-1}\,dx\,dt',
\end{split}
 \eeq
where~$v_{\rm curl}^{\rm h}$ (resp.~$v_{\rm div}^{\rm h}$)
corresponds to the horizontal divergence free (resp. curl free) part of
the horizontal vector~$v^{\rm h}=(v^1,v^2),$ which is given by
\eqref{a.1wrt},  and where~$\omega_{r-1}\eqdefa |\omega|^{r-2}\,\omega$.

Let us start with the easiest term~$F_1$. 
 We first get, by using
integration by parts, that \beno |F_1(t)|&\leq&
r\int_0^t\int_{\R^3}|v^3(t',x)|\,|\pa_3\omega(t',x)|\,|\omega(t',x)|^{r-1}\,dx\,dt'\\
&\leq&
r\int_0^t\int_{\R^3}|v^3(t',x)|\,|\pa_3\omega(t',x)|\,|\omega_{\frac{r}2}(t',x)|^{\frac 2 {r'}   }\,dx\,dt'.
\eeno Notice that
$$
\frac {p-2} {3p} + \frac1r +
\frac{2pr-3p+2r}{6p(r-1)}\times\frac 2 {r'} =1,
$$
we get, by applying H\"older inequality, that
$$
|F_1(t)|\leq r
\int_0^t\|v^3(t')\|_{L^{\frac{3p}{p-2}}}\|\pa_3\omega(t')\|_{L^{r}}\bigl\|\omega_{\frac{r}2}(t')\bigr\|_{L^{\frac{6p(r-1)}{2pr-3p+2r}}}^{\frac 2 {r'}   }\,dt'.
$$
As~$p$ is in~$\ds\bigl]4,\frac{2r}{2-r}\bigr[$, we observe that
$\ds r'\frac{p-2}{2p}$ belongs to~$]0,1[.$ Then  Sobolev embedding and
interpolation inequality implies that
$$
\bigl\|\omega_{\frac{r}2}(t')\bigr\|_{L^{\frac{6p(r-1)}{2pr-3p+2r}}}
\lesssim
\bigl\|\omega_{\frac{r}2}  (t')\bigr\|_{\dH^{r'\frac{(p-2)}{2p}}}\lesssim
\bigl\|\omega_{\frac{r}2}(t')\bigr\|_{L^2}^{\frac{2r-p(2-r)}{2p(r-1)}}
\bigl\|\nabla \omega_{\frac{r}2}(t')\bigr\|_{L^2}^{r'\frac{p-2}{2p}},
$$
from which and \refeq{estimbasomega34}, we infer
$$
{ |F_1(t)| \lesssim \int_0^t\|v^3(t')\|_{L^{\frac{3p}{p-2}}}
\bigl\|\pa_3\omega_{\frac{r}2}(t')\bigr\|_{L^2}
\bigl\|\omega_{\frac{r} 2}(t')\bigr\|_{L^2}^{\frac 2 {r}   -1} }
{
\bigl\|\na\omega_{\frac{r}2}(t')\bigr\|_{L^2}^{1-\frac 2 p }\bigl\|\omega_{\frac{r}2}(t')\bigr\|_{L^2}^{1-2\left(\frac 1 {r}   -\frac 1 p \right)}\,dt'.
}
$$
Applying Young's inequality gives rise to
\begin{equation}\begin{split}\label{F1}
|F_1(t)| & \lesssim
\int_0^t\|v^3(t')\|_{L^{\frac{3p}{p-2}}}\bigl\|\omega_{\frac{r}2}(t')\bigr\|_{L^2}^{\frac 2 p }\bigl\|\na\omega_{\frac{r}2}(t')\bigr\|_{L^2}^{2\left(1-\f1p\right)} \,dt' \nonumber\\
& \leq
\frac{r-1}{r^2}\int_0^t\bigl\|\na\omega_{\frac{r}2}(t')\bigr\|_{L^2}^2\,dt'
+C\int_0^t\|v^3(t')\|_{L^{\frac{3p}{p-2}}}^p\bigl\|\omega_{\frac{r}2}(t')\bigr\|_{L^2}^{2}\,dt'.
\end{split}\end{equation}

\medbreak The other two terms in \eqref{estimateprop1}
require a refined way to describe the regularity
of~$\omega_{\frac r2}$ and demand a detailed study  of the anisotropic
operator~$\nabla_{\rm h}\D_{\rm h}^{-1}$ associated with the
Biot-Savart's law in horizontal variables. We first modify
Lemma 4.1 of  \cite{CZZ} to the following one.

 \begin{prop}
 \label{S4prop1}
{\sl Under the assumptions of Proposition \ref{prop2.1} and let $\s=r'\bigl(\f12-\f1p\bigr)$,
we have
\begin{equation}\begin{split}
 \label{b.1}
 \Bigl| \int_{\R^3}
\partial_{\rm h}\D_{\rm h}^{-1}f
\cdot &\partial_{\rm h}a\,\omega_{r-1} dx\Bigr|
\lesssim  \min\bigl\{ \|f\|_{L^{r}},  \| f\|_{\cH^{\theta,r}}\bigr\}\|a\|_{\cB_{q_1,q_2,r}^{\mu,p}}
\bigl\|\omega_{\frac{r}2}\bigr\|_{\dH^\s}^{\frac 2 {r'}},
\end{split}\end{equation} 
where the norm $\|\cdot\|_{\cB_{q_1,q_2,r}^{\mu,p}}$ is given by \eqref{nal}.
} \end{prop}

\begin{proof}
Observe that $\ds \omega_{r-1} =G(\omega_{\frac{r}2})$ with $G(z)\eqdefa z|z|^{-2\alpha(r)}$. It follows from  Lemma\refer{puisancealphaBesov} that
 \begin{equation}
 \label{anisotropicomega}
\bigl\|\omega_{r-1}\bigr\|_{\dB^{\frac {2\s} {r'} }_{r'   ,r'   }}
\lesssim \bigl\|\omega_{\frac{r}2}\bigr\|_{\dH^\s}^{\frac 2 {r'}   }\quad \forall \s\in ]0,1[.
 \end{equation}
Let us study the product $\partial_{\rm h}a\,\omega_{r-1}.$ By applying Bony's decomposition in the horizontal variables, we
 write
 \beno
 \partial_{\rm h} a\, \omega_{r-1} & = & \Th({\partial_{\rm h} a}, \omega_{r-1}) +\Rh(\partial_{\rm h}a,\omega_{r-1})+
 \Th({\omega_{r-1}}, \partial_{\rm h}a)\\
 & = & \partial_{\rm h}\Th({\omega_{r-1}}, a) +A(a,\omega)\with \\
 A(a,\omega) &\eqdefa&
  \Th({\partial_{\rm h} a}, \omega_{r-1}) +\Rh(\partial_{\rm h}a,\omega_{r-1})
 -\Th({ \partial_{\rm h}\omega_{r-1}},a).
 \eeno
In view of Lemma \ref{lemBern}, it is obvious that we only need to prove
\eqref{b.1} for $q_1\in[r,2[$ and $q_2\in\bigl]r',\bigl(\frac1p+3\al(r)+\mu\bigr)^{-1}\bigr[$.
Then we can estimate the above term by term as follows:
\begin{align}
\|&\Th(\omro,a)\|_{L^{r'}}+\|\Th(\omro,a)\|_{\dH^{3\alpha(r)-\theta,\theta}}\lesssim
\|a\|_{\cB_{2,q_2,r}^{\mu,p}}
\bigl\|\omega_{\frac{r}2}\bigr\|_{\dH^\s}^{\frac 2 {r'}},\label{S4eq1}\\
\bigl\|\bigl(\Th(\partial_{\rm h}& a, \omega_{r-1}),\Th({\partial_{\rm h}\omega_{r-1} }, a)\bigr)
\bigr\|_{\bigl(B^{\mu-\d_1}_{\f{2r'}{r'+2},2}\bigr)_\h \bigl(H^{\d_1+\al(r)-\mu}\bigr)_{\v}}\lesssim
\|a\|_{\cB_{2,q_2,r}^{\mu,p}}
\bigl\|\omega_{\frac{r}2}\bigr\|_{H^\s}^{\f2{r'}},\label{S4eq5}\\
\|&R^\h({\partial_{\rm h} a},\omro)\|_{\bigl(B^{\mu+\f2{q_1}-1-\d_2}_{\f{q_1r'}{q_1+r'},2}\bigr)_\h\bigl(H^{\d_2+\al(r)-\mu}\bigr)_\v}\lesssim
 \|a\|_{\cB_{q_1,q_2,r}^{\mu,p}}\bigl\|\omega_{\f{r}2}\bigr\|_{H^\s}^{\f2{r'}},\label{S4eq6}
\end{align}
where $\d_1\in \bigl]\mu-\al(r), 1-2/p\bigr[$ and $\d_2\in \bigl]0,\min\bigl(1-\f2p,\mu-1+2/{q_1}\bigr)\bigr[$.
The proofs of \eqref{S4eq1}-\eqref{S4eq6} will be postponed to Appendix \ref{apA}.
Let us continue our proof of the proposition.

Note that $q_1\in ]1,2[$, we have $\cB_{q_1,q_2,r}^{\mu,p}\hookrightarrow\cB_{2,q_2,r}^{\mu,p}$.
Then we deduce from \eqref{S4eq1}, that
\begin{equation}\begin{split}\label{conclusion1}
\bigl|\int_{\R^3}\ph\Laplacianh f\cdot&\ph\Th(\omro,a)dx\bigr|=\bigl|\int_{\R^3}\ph^2\Laplacianh f\cdot\Th(\omro,a)dx\bigr|\\
&\lesssim\min\left\{\|f\|_{L^r}\|\Th(\omro,a)\|_{L^{r'}},\|f\|_{\htr}\|\Th(\omro,a)\|_{\dH^{3\alpha(r)-\theta,\theta}}\right\}\\
&\lesssim\min\bigl\{\|f\|_{L^r},\|f\|_{\htr}\bigr\}\|a\|_{\cB_{q_1,q_2,r}^{\mu,p}}
\|\omr\|_{\dH^{\sigma}}^{\frac{2}{r'}}.
\end{split}\end{equation}

Whereas for any $r\in ]1,2[,$ by using Minkowski's inequality and Lemma \ref{Thm2.40BCD} twice, we have
\begin{equation}\begin{split}\label{S4eq7}
\|f\|_{\dB_{r,2}^{0,0}}
\lesssim &\Bigl(\sum_{k\in\Z}\bigl\|\bigl(\sum_{\ell\in\Z}\|\D_\ell^\v\D_k^\h f\|_{L^r_\v}^2\bigr)^{\f12}\bigr\|_{L^r_\h}^2\Bigr)^{\f12}\\
\lesssim &\Bigl(\sum_{k\in\Z}\bigl\|\|\D_k^\h f\|_{L^r_\v}\bigr\|_{L^r_\h}^2\Bigr)^{\f12}\\
\lesssim &\Bigl\|\bigl(\sum_{k\in\Z}\|\D_k^\h f\|_{L^r_\h}^2\bigr)^{\f12}\Bigr\|_{L^r_\v}\lesssim \|f\|_{L^r}.
\end{split}\end{equation}
And it follows from Lemma \ref{lemBern} once again that
\beno
\bigl(\dB^{0}_{\f{2r'}{r'+2},2}\bigr)_\h \bigl(\dH^{\al(r)}\bigr)_{\v}
\hookrightarrow \dB^{-1,0}_{r',2},
\ \bigl(\dB^{\f2{q_1}-1}_{\f{q_1r'}{q_1+r'},2}\bigr)_\h\bigl(\dH^{\al(r)}\bigr)_\v
\hookrightarrow \dB^{-1,0}_{r',2}.
\eeno
Using \eqref{S4eq5},~\eqref{S4eq6} with $\d_1=\d_2=\mu$, and \eqref{S4eq7},
 we achieve
\begin{equation*}\begin{split}
\bigl|\int_{\R^3}&\ph\Laplacianh f\cdot\bigl(\Th(\ph a,\omro)+\Th(\ph \omro, a)+\Rh(\ph a,\omro)\bigr)\,dx\bigr|\\
\leq &\|\ph\Laplacianh f\|_{\dB^{1,0}_{r,2}}
\bigl\|\bigl(\Th(\ph a,\omro), \Th( \ph\omro, a),\Rh(\ph a,\omro)\bigr)\bigr\|_{\dB^{-1,0}_{r',2}}\\
\lesssim &\|f\|_{L^r}\|a\|_{\cB_{q_1,q_2,r}^{\mu,p}}
\|\omr\|_{\dH^{\sigma}}^{\frac{2}{r'}}.
\end{split}\end{equation*}
Combining the above estimate with \eqref{conclusion1}, we conclude that
\begin{equation}\label{conclusion}
 \Bigl| \int_{\R^3}
\partial_{\rm h}\D_{\rm h}^{-1}f\cdot
\partial_{\rm h}a \,\omega_{r-1} dx\Bigr|
\lesssim \|f\|_{L^r}\|a\|_{\cB_{q_1,q_2,r}^{\mu,p}}
\|\omr\|_{\dH^{\sigma}}^{\frac{2}{r'}}.
\end{equation}

On the other hand, it follows from Lemma \ref{lemBern} once again that
\beno
\bigl(\dB^{\al(r)-\th}_{\f{2r'}{r'+2},2}\bigr)_\h \bigl(\dH^{\th}\bigr)_{\v}
\hookrightarrow \dH^{-1+3\al(r)-\th,\th},\
\bigl(\dB^{\f2{q_1}-1+\al(r)-\th}_{\f{q_1r'}{q_1+r'},2}\bigr)_\h\bigl(\dH^{\th}\bigr)_\v
\hookrightarrow \dH^{-1+3\al(r)-\th,\th}.
\eeno
Using \eqref{S4eq5},~\eqref{S4eq6} with
$\d_1=\d_2=\mu+\th-\al(r)$ yields
\begin{equation*}\begin{split}
&\bigl|\int_{\R^3}\ph\Laplacianh f\cdot\bigl(\Th(\ph a,\omro)+\Th(\ph\omro,a)+\Rh(\ph a,\omro)\bigr)\,dx\bigr|\\
&\leq\|\ph\Laplacianh f\|_{\dH^{1-3\alpha(r)+\theta,-\theta}}
\bigl\|\bigl(\Th(\ph a,\omro),\Th(\ph \omro,a),\Rh(\ph a,\omro)\bigr)\bigr\|_{\dH^{-1+3\alpha(r)-\theta,\theta}}\\
&\lesssim\|f\|_{\htr}\|a\|_{\cB_{q_1,q_2,r}^{\mu,p}}\|\omega_{\frac{r}2}\bigr\|_{\dH^\s}^{\frac 2 {r'}},
\end{split}\end{equation*}
which together with \eqref{conclusion1} gives rise to
\begin{equation}\label{conclusion1}
 \Bigl| \int_{\R^3}
\partial_{\rm h}\D_{\rm h}^{-1}f
\partial_{\rm h}a \,\omega_{r-1} dx\Bigr|
\lesssim \|f\|_{\htr}\|a\|_{\cB_{q_1,q_2,r}^{\mu,p}}
\bigl\|\omega_{\frac{r}2}\bigr\|_{\dH^\s}^{\frac 2 {r'}}.
\end{equation}

Combining the Estimates \eqref{conclusion} and \eqref{conclusion1}, we complete the proof of this proposition.
\end{proof}

\medbreak  The estimate of~$F_2(t)$ uses the Biot-Savart's law in the
horizontal variables  (namely\refeq{a.1wrt}) and
 Proposition \ref{S4prop1} with~$f=\partial_3\omega$,~$a=v^3$ and
$\sigma=\frac{(p-2)r'}{2p}$, which is in $]\frac{r'}{4},1[$ provided $p\in]4,\frac{2r}{2-r}[$. This
gives for any time~$t<T^\star$ that
\begin{equation}\begin{split}\label{Iomega}
I_\omega(t)&\eqdefa\Bigl|\int_{\R^3}\bigl(\pa_2v^3(t,x)\pa_3v_{\rm
curl}^1(t,x)-\pa_1v^3(t,x)\pa_3v_{\rm curl}^2(t,x) \bigr)
\omega_{r-1}(t,x)\,dx\Bigr|\\
&\lesssim  \|v^3(t)\|_{\cB_{q_1,q_2,r}^{\mu,p}}\|\pa_3\omega(t) \|_{L^{r}}
\bigl\|\omega_{\frac{r}2}(t)\bigr\|_{\dH^{\frac{(p-2)r'}{2p}}}^{\frac 2 {r'}}.
\end{split}\end{equation}
By virtue of \eqref{estimbasomega34} and of the interpolation
inequalities between~$L^2$ and~$\dH^1$,
\eqref{Iomega} implies
\begin{equation*}\begin{split}
I_\omega(t)&\lesssim  \|v^3(t)\|_{\cB_{q_1,q_2,r}^{\mu,p}}
\bigl\|\omr(t)\bigr\|_{L^2}^{\frac2r-1}\bigl\|\nabla\omr(t)\bigr\|_{L^2}
\bigl\|\omr(t)\bigr\|_{L^2}^{\frac2{r'}-\frac{p-2}{p}}\bigl\|\nabla\omr(t)\bigr\|_{L^2}^{\frac{p-2}{p}}\\
&\lesssim \|v^3(t)\|_{\cB_{q_1,q_2,r}^{\mu,p}}
\bigl\|\omr(t)\bigr\|_{L^2}^{\frac2p}\bigl\|\nabla\omr(t)\bigr\|_{L^2}^{2(1-\frac1p)}.
\end{split}\end{equation*}
Then by using Young's inequality and integrating in time, we get
\begin{equation}\label{F2}
|F_2(t)|\leq \frac{r-1}{r^2}\int_0^t \bigl\|\nabla\omr(t')\bigr\|_{L^2}^{2} dt'
+C\int_0^t\|v^3(t')\|_{\cB_{q_1,q_2,r}^{\mu,p}}^p
\bigl\|\omr(t')\bigr\|_{L^2}^{2}dt',
\end{equation}

\medbreak The estimate of~$F_3(t)$ uses \eqref{a.1wrt} and
Proposition  \refer{S4prop1} with~$f=\partial_3^2 v^3$,~$a=v^3$:
\begin{equation*}\begin{split}
|F_3(t)|&=\Bigl|-\int_0^t\int_{\R^3}\bigl(\pa_2 v^3(t')\cdot\pa_1\Laplacianh\pa_3^2 v^3(t')
-\pa_1 v^3(t')\cdot\pa_2\Laplacianh\pa_3^2 v^3(t')\bigr)\omro(t')dxdt'\Bigr|\\
&\lesssim\int_0^t\|\pa_3^2 v^3(t')\|_{\htr}\|v^3(t')\|_{\cB_{q_1,q_2,r}^{\mu,p}}\bigl\|\omr(t')\bigr\|_{\dH^{\frac{(p-2)r'}{2p}}}^{\frac{2}{r'}}dt'\\
&\lesssim\int_0^t\|\pa_3^2 v^3\|_{\htr}\|v^3\|_{\cB_{q_1,q_2,r}^{\mu,p}}^{p\alpha(r)}
\bigl(\|v^3\|_{SC}^p \bigl\|\omr\|_{L^2}^2\bigr)^{\frac1p-\alpha(r)}
\bigl\|\nabla\omr\bigr\|_{L^2}^{2\bigl(\frac12-\frac1p\bigr)}dt'.
\end{split}\end{equation*}
Applying H\"{o}lder's inequality and then Young's inequality leads to
\begin{equation}\begin{split}\label{F3}
|F_3(t)|&\leq C\Bigl(\int_0^t\|\pa_3^2 v^3\|_{\htr}^2dt'\Bigr)^{\frac12}
\Bigl(\int_0^t\|v^3\|_{\cB_{q_1,q_2,r}^{\mu,p}}^{p}dt'\Bigr)^{\alpha(r)}\\
&\qquad\qquad\times\Bigl(\int_0^t\|v^3\|_{\cB_{q_1,q_2,r}^{\mu,p}}^p \bigl\|\omr\|_{L^2}^2dt'\Bigr)^{\frac1p-\alpha(r)}
\Bigl(\int_0^t\bigl\|\nabla\omr\bigr\|_{L^2}^{2}dt'\Bigr)^{\frac12-\frac1p}\\
&\leq \frac{r-1}{r^2}\int_0^t\bigl\|\nabla\omr\bigr\|_{L^2}^{2}dt'
+C\int_0^t\|v^3\|_{\cB_{q_1,q_2,r}^{\mu,p}}^p \bigl\|\omr\|_{L^2}^2dt'\\
&\qquad\qquad+C\Bigl(\int_0^t\|\pa_3^2 v^3\|_{\htr}^2dt'\Bigr)^{\frac{r}2}
\Bigl(\int_0^t\|v^3\|_{\cB_{q_1,q_2,r}^{\mu,p}}^{p}dt'\Bigr)^{1-\frac{r}2}.
\end{split}\end{equation}

Substituting the estimates \eqref{F1}, \eqref{F2} and \eqref{F3} into \eqref{estimateprop1}, we obtain
\begin{equation}\begin{split}
\frac 1r \bigl\|& \omega_{\frac{r}2}(t) \bigr\|_{L^2}^2 +
\frac{r-1}{r^2}\int_0^t \bigl\|\nabla
\omega_{\frac{r}2}(t')\bigr\|_{L^2}^2dt' \leq
\frac1r\bigl\||\omega_0|^{\frac{r}2}\bigr\|_{L^2}^2 \\
&+C\int_0^t\|v^3\|_{SC}^p\bigl\|\omega_{\frac{r}2}\bigr\|_{L^2}^2 dt'+C\Bigl(\int_0^t\|\pa_3^2 v^3\|_{\htr}^2dt'\Bigr)^{\frac{r}2}
\Bigl(\int_0^t\|v^3\|_{\cB_{q_1,q_2,r}^{\mu,p}}^{p}dt'\Bigr)^{1-\frac{r}2}.
\end{split}\end{equation}
Then using Gronwall's inequality and the elementary inequality that $x^{1-\frac{r}{2}} e^{Cx}\lesssim e^{C'x}$
for some constant $C'>C$ and any $x\geq0$ yields \eqref{prop1},
which is the desired result.


\setcounter{equation}{0}
\section{Proof of the  estimate for $\p3^2 v^3$}
\label{Sect4}

In this section, we shall present the proof of Proposition
\ref{prop2.2}.  Recall the  $\pa_3v^3$ equation of $(\wt{NS})$ that
\beq \label{S5eq0}
\partial_t \partial_3v^3 +v\cdot\nabla\partial_3v^3-\D\partial_3v^3+\partial_3v\cdot\nabla v^3
=-\partial_3^2\D^{-1} \Bigl(\ds\sum_{\ell,m=1}^3 \partial_\ell
v^m\partial_mv^\ell\Bigr).
\eeq
Let $\htr$ be given by
Definition\refer{defhtr}. Taking $\cH^{\theta,r}$ inner
product of the \eqref{S5eq0}  with $\pa_3v^3,$ gives
\begin{equation}\begin{split} \label{5.1}
\f12\f{d}{dt}\|\pa_3v^3(t)\|&_{\cH^{\theta,r}}^2+\|\na\pa_3v^3(t)\|_{\cH^{\theta,r}}^2 =-\sum_{n=1}^3 \bigl(Q_n(v,v) \, |\, \pa_3v^3\bigr)_{\cH^{\theta,r}}\with\\
 Q_1(v,v) & \eqdefa \bigl(\Id+\partial_3^2\Delta^{-1} \bigr)(\partial_3v^3)^2 +\partial_3^2\Delta^{-1}
  \biggl(\sum_{\ell,m=1}^2\partial_\ell v^m\partial_m  v^\ell\biggr) \,,\\
Q_2(v,v) \eqdefa  \bigl(& \Id+2\partial_3^2\Delta^{-1} \bigr)\biggl(\sum_{\ell=1}^2 \partial_3 v^\ell
 \partial_\ell v^3\biggr)\andf
 Q_3(v,v) \eqdefa v\cdot \nabla\partial_3 v^3.
 \end{split}\end{equation}

 \no$\bullet$ \underline{The estimate of $\bigl(Q_1(v,v) \, |\, \pa_3v^3\bigr)_{\cH^{\theta,r}}$}

The estimate of this term relies on the following lemma:
 \begin{lem}
\label{lem5.1}
{\sl Let~$L(D)$ be an $L^q$ bounded Fourier multiplier for any $q\in]1,\infty[$.
Let $r\in\bigl]3/2,2\bigr[,~\theta\in]0,\alpha(r)[,~p\in]4,\infty[$,
and $s_1,~s_2\in ]1,\infty[$ satisfy
\begin{equation}\label{lem5.1.2}
\frac{2}{s_1}+\frac1{s_2}=\frac1{p'}+3\al(r) \andf \th<\frac1{s_2}<\frac1{p'}-3\al(r)+\th.
\end{equation}
Then we have
\begin{equation}\begin{split}\label{lem5.1.3}
 \bigl|\bigl(L(D) &(fg) \,|\,\partial_3v^3\bigr)_{\cH^{\theta,r}}\bigr|\lesssim\|f\|_{\htrps}\|g\|_{\htrps}\cdot\|v^3\|_{\Bs},
\end{split}\end{equation}
where we denote $\|f\|_{\htrps}\eqdefa \|f\|_{\dH^{\theta,\frac1{p'}-3\al(r)-\theta}}+
\|f\|_{\dH^{\frac1{p'}-\frac1{s_2}-3\alpha(r)+\theta,\frac1{s_2}-\theta}}$.
} \end{lem}

\begin{proof}
Recall that   $\htr=\bigl(\dB^{-3\al(r)+\th}_{2,2}\bigr)_{\rm{h}}\bigl(\dB^{-\th}_{2,2}\bigr)_{\rm{v}}$,
we write
$$
\bigl(L(D) (fg) \,\big|\,\partial_3v^3\bigr)_{\cH^{\theta,r}}=
\sum\limits_{k,\ell\in\Z}2^{2k(-3\al(r)+\th)}2^{-2\ell\th}\Bigl(\dhk\dvl\bigl(L(D) (fg)\bigr)\,\Big|\,\dhk\dvl\pa_3 v^3\Bigr)_{L^2}.
$$
Applying Lemma \ref{lemBern} yields
\begin{equation*}\begin{split}
\bigl|\bigl(L(D)(fg) \,|\,\partial_3v^3\bigr)_{\cH^{\theta,r}}\bigr|
&\lesssim\sum\limits_{k,\ell\in\Z}2^{2k(-3\al(r)+\th)}2^{-2\ell\th}\|\dhk\dvl(fg)\|
_{L_{\rm{h}}^{\frac{s_1}{s_1-1}}L_{\rm{v}}^{\frac{s_2}{s_2-1}}}
2^\ell\|\dhk\dvl v^3\|_{L_{\rm{h}}^{s_1}L_{\rm{v}}^{s_2}}\\
&\lesssim \|fg\|_{\Bsfg}\|v^3\|_{\Bs}.
\end{split}\end{equation*}
So that it remains to verify
\begin{equation}\label{lem5.1.4}
\|fg\|_{\Bsfg}\lesssim\|f\|_{\htrps}\|g\|_{\htrps}.
\end{equation}
In order to do so, we get, applying Bony's decomposition in both horizontal and vertical variables, that
\beq\label{S5eq1}
fg=\left(T^\h+R^\h+\bar{T}^\h\right)\left(T^\v+R^\v+\bar{T}^\v\right)(f,g).
\eeq
We first get, by applying Lemma \ref{lemBern} and \eqref{lem5.1.2}, that
\beq\label{S5eq2}
\begin{split}
\|\D_{k}^\h S_{\ell-1}^\v f\|_{L^2_\h(L^{\f{2s_2}{s_2-2}}_\v)}\lesssim &\sum_{\ell'\leq\ell-2}c_{k,\ell'}2^{-k\left(\f1{p'}-\f1{s_2}-3\al(r)+\th\right)}
2^{\ell'\th}\|f\|_{\dH^{\frac1{p'}-\frac1{s_2}-3\alpha(r)+\theta,\frac1{s_2}-\theta}}\\
\lesssim &c_{k,\ell}2^{-k\left(\f1{p'}-\f1{s_2}-3\al(r)+\th\right)}
2^{\ell\th}\|f\|_{\dH^{\frac1{p'}-\frac1{s_2}-3\alpha(r)+\theta,\frac1{s_2}-\theta}},
\end{split}
\eeq
and
\beq\label{S5eq3}
\begin{split}
\|S_{k-1}^\h \D_{\ell}^\v f\|_{L^{\f{2s_1}{s_1-2}}_\h(L^2_\v)}\lesssim &\sum_{k'\leq k-2}c_{k',\ell}2^{k'\left(\f2{s_1}-\th\right)}
2^{-\ell\left(\f1{p'}-3\al(r)-\th\right)}\|f\|_{\dH^{\th,\frac1{p'}-3\alpha(r)-\theta}}\\
\lesssim &c_{k,\ell}2^{k\left(\f2{s_1}-\th\right)}2^{-\ell\left(\f1{p'}-3\al(r)-\th\right)}
\|f\|_{\dH^{\th,\frac1{p'}-3\alpha(r)-\theta}}.
\end{split}
\eeq
And applying Lemma \ref{lemBern} and \eqref{S5eq2} gives rise to
\ben\nonumber
\|S_{k-1}^\h S_{\ell-1}^\v f\|_{L^{\f{2s_1}{s_1-2}}_\h
(L^{\f{2s_2}{s_2-2}}_\v)}&\lesssim &\sum_{k'\leq k-2}2^{\f{2k'}{s_1}}\|\D_{k'}^\h S_{\ell-1}^\v f\|_{L^2_\h(L^{\f{2s_2}{s_2-2}}_\v)}\\
&\lesssim &\sum_{k'\leq k-2}c_{k',\ell}2^{k'\left(\f2{s_1}+\f1{s_2}-\f1{p'}+3\al(r)-\th\right)}
2^{\ell\th}\|f\|_{\dH^{\frac1{p'}-\frac1{s_2}-3\alpha(r)+\theta,\frac1{s_2}-\theta}} \label{S5eq4}\\
&\lesssim &c_{k,\ell}2^{k(6\al(r)-\th)}2^{\ell\th}\|f\|_{\dH^{\frac1{p'}-\frac1{s_2}-3\alpha(r)+\theta,\frac1{s_2}-\theta}}.\nonumber
\een
Considering the support to the Fourier transform of the terms in $T^\h T^\v(f,g),$ we have
\beno
\begin{split}
\|\D_k^\h\D_\ell^\v \bigl(T^\h+\bar{T}^\h\bigr) T^\v(f,g)\|_{L^{s_1'}_\h(L^{s_2'}_\v)}\lesssim & \sum_{\substack{|k'-k|\leq 4\\|\ell'-\ell|\leq 4}}\Bigl(\|S_{k'-1}^\h S_{\ell'-1}^\v f\|_{L^{\f{2s_1}{s_1-2}}_\h
(L^{\f{2s_2}{s_2-2}}_\v)}\|\D_{k'}^\h\D_{\ell'}^\v g\|_{L^2}\\
&\qquad\quad+\|\D_{k'}^\h S_{\ell'-1}^\v f\|_{L^2_\h(L^{\f{2s_2}{s_2-2}}_\v)}\|S_{k'-1}^\h\D_{\ell'}^\v g\|_{L^{\f{2s_1}{s_1-2}}_\h(L^2_\v)}\Bigr)\\
\lesssim & d_{k,\ell}2^{2k(3\al(r)-\th)}2^{-\ell\left(\f1{p'}-3\al(r)-2\th\right)}\\
&\qquad\times\|f\|_{\dH^{\frac1{p'}-\frac1{s_2}-3\alpha(r)+\theta,\frac1{s_2}-\theta}}\|g\|_{\dH^{\th,\frac1{p'}-3\alpha(r)-\theta}}.
\end{split}
\eeno
By symmetry, we obtain
 \beno
\begin{split}
\|\D_k^\h\D_\ell^\v \bigl(T^\h+\bar{T}^\h\bigr) \bar{T}^\v(f,g)\|_{L^{s_1'}_\h(L^{s_2'}_\v)}
\lesssim & d_{k,\ell}2^{2k(3\al(r)-\th)}2^{-\ell\left(\f1{p'}-3\al(r)-2\th\right)}\\
&\qquad\times\|f\|_{\dH^{\th,\frac1{p'}-3\alpha(r)-\theta}}\|g\|_{\dH^{\frac1{p'}-\frac1{s_2}-3\alpha(r)+\theta,\frac1{s_2}-\theta}}.
\end{split}
\eeno
While we deduce from Lemma \ref{lemBern} that
\beno
\begin{split}
\|\D_k^\h\D_\ell^\v R^\h T^\v(f,g)\|_{L^{s_1'}_\h(L^{s_2'}_\v)}\lesssim &2^{\f{2k}{s_1}}\sum_{\substack{k'\geq k-3\\|\ell'-\ell|\leq 4}}
\|\D_{k'}^\h S_{\ell'-1}^\v f\|_{L^2_\h(L^{\f{2s_2}{s_2-2}}_\v)}\|\wt{\D}_{k'}^\h\D_{\ell'}^\v g\|_{L^2},
\end{split}
\eeno
from which, \ref{lem5.1.2} and \eqref{S5eq2}, we deduce that
\beno
\begin{split}
\|\D_k^\h\D_\ell^\v R^\h T^\v(f,g)\|_{L^{s_1'}_\h(L^{s_2'}_\v)}\lesssim &2^{\f{2k}{s_1}}\sum_{\substack{k'\geq k-3\\|\ell'-\ell|\leq 4}}
d_{k',\ell'}2^{-k'\left(\f1{p'}-\f1{s_2}-3\al(r)+2\th\right)}2^{-\ell'\left(\f1{p'}-3\al(r)-2\th\right)}\\
&\qquad\times \|f\|_{\dH^{\frac1{p'}-\frac1{s_2}-3\alpha(r)+\theta,\frac1{s_2}-\theta}}\|g\|_{\dH^{\th,\frac1{p'}-3\alpha(r)-\theta}}\\
\lesssim & d_{k,\ell}2^{2k(3\al(r)-\th)}2^{-\ell\left(\f1{p'}-3\al(r)-2\th\right)}\|f\|_{\htrps}\|g\|_{\htrps}.
\end{split}
\eeno
By symmetry, the term $\|\D_k^\h\D_\ell^\v R^\h \bar{T}^\v(f,g)\|_{L^{s_1'}_\h(L^{s_2'}_\v)}$ shares the above estimate.

Again we deduce from Lemma \ref{lemBern} that
\beno
\begin{split}
\|\D_k^\h\D_\ell^\v T^\h R^\v(f,g)\|_{L^{s_1'}_\h(L^{s_2'}_\v)}\lesssim &2^{\f{\ell}{s_2}}\sum_{\substack{|k'- k|\leq 4\\\ell'\geq\ell- 3}}
\|S_{k'-1}^\h \D_{\ell'}^\v f\|_{L^{\f{2s_1}{s_1-2}}_\h(L^2_\v)}\|{\D}_{k'}^\h\wt{\D}_{\ell'}^\v g\|_{L^2},
\end{split}
\eeno
which together with \eqref{lem5.1.2} and \eqref{S5eq3} ensures that
\beno
\begin{split}
\|\D_k^\h\D_\ell^\v T^\h R^\v(f,g)\|_{L^{s_1'}_\h(L^{s_2'}_\v)}\lesssim &2^{\f{\ell}{s_2}}\sum_{\substack{|k'- k|\leq 4\\\ell'\geq\ell- 3}}
d_{k',\ell'}2^{-k\left(\f1{p'}-\f2{s_1}-\f1{s_2}-3\al(r)+2\th\right)}2^{-\ell\left(\f1{p'}+\f1{s_2}-3\al(r)-\th\right)}\\
&\qquad\times \|f\|_{\dH^{\th,\frac1{p'}-3\alpha(r)-\theta}}
\|g\|_{\dH^{\frac1{p'}-\frac1{s_2}-3\alpha(r)+\theta,\frac1{s_2}-\theta}}\\
\lesssim & d_{k,\ell}2^{2k(3\al(r)-\th)}2^{-\ell\left(\f1{p'}-3\al(r)-2\th\right)}\|f\|_{\htrps}\|g\|_{\htrps}.
\end{split}
\eeno
By symmetry, the same estimate holds for $\D_k^\h\D_\ell^\v \bar{T}^\h R^\v(f,g).$

Finally, we get, by applying Lemma \ref{lemBern} and \eqref{lem5.1.2}, that
\beno
\begin{split}
\|\D_k^\h\D_\ell^\v R^\h R^\v(f,g)\|_{L^{s_1'}_\h(L^{s_2'}_\v)}\lesssim &2^{\f{2k}{s_2}} 2^{\f{\ell}{s_2}}\sum_{\substack{k'\geq k-3\\\ell'\geq\ell- 3}}
\|\D_{k'}^\h \D_{\ell'}^\v f\|_{L^2}\|\wt{\D}_{k'}^\h\wt{\D}_{\ell'}^\v g\|_{L^2}\\
\lesssim &2^{\f{2k}{s_2}} 2^{\f{\ell}{s_2}}\sum_{\substack{k'\geq k-3\\\ell'\geq\ell- 3}}d_{k',\ell'}2^{-k\left(\f1{p'}-\f1{s_2}-3\al(r)+2\th\right)}2^{-\ell\left(\f1{p'}+\f1{s_2}-3\al(r)-2\th\right)}\\
&\qquad\times \|f\|_{\dH^{\th,\frac1{p'}-3\alpha(r)-\theta}}
\|g\|_{\dH^{\frac1{p'}-\frac1{s_2}-3\alpha(r)+\theta,\frac1{s_2}-\theta}}\\
\lesssim & d_{k,\ell}2^{2k(3\al(r)-\th)}2^{-\ell\left(\f1{p'}-3\al(r)-2\th\right)}\|f\|_{\htrps}\|g\|_{\htrps}.
\end{split}
\eeno

By summing up the above estimates, we obtain \eqref{lem5.1.4}, and thus the lemma.
\end{proof}

Applying Lemma \ref{lem5.1}
with $f$ and $g$ being of the forms $\ph\vhcurl,~\ph\vhdiv$ or $\p3 v^3$ gives:
\begin{equation}\label{5.21}
\bigl|\bigl(Q_1(v,v)|\p3 v^3\bigr)_{\htr}\bigr|\lesssim\|v^3\|_{\Bs}
\bigl(\|\omega\|_{\htrps}^2
+\|\p3 v^3\|_{\htrps}^2\bigr).
\end{equation}
Due to $s_2$ satisfying \eqref{lem5.1.2}, we get, by
applying Lemma \ref{lem4.3ofCZ5} and \ref{BiotSavartomega}, that
\begin{equation}\label{5.22}
\|\omega\|_{\htrps}\lesssim\|\omega\|_{\dH^{\frac1{p'}-3\alpha(r)}}
\lesssim\bigl\|\omr\bigr\|_{L^2}^{2\alpha(r)+\frac1p}\bigl\|\nabla\omr\bigr\|_{L^2}^{1-\frac1{p}}.
\end{equation}
While for any function $a$, it follows from Definition \ref{defhtr} that
\begin{align*}
\|a\|_{\dH^{\theta,\frac1{p'}-3\al(r)-\theta}}^2
&=\int_{\R^3}\bigl|\widehat{a}(\xi)\bigr|^{2}\bigl(|\xih|^{6\al(r)}
|\xi_3|^{2(\frac1{p'}-3\al(r))}\bigr)
\cdot|\xih|^{2(-3\alpha(r)+\theta)}|\xi_3|^{-2\theta}\,d\xi\\
&\leq \int_{\R^3}\bigl|\widehat{a}(\xi)\bigr|^{\frac2p}\bigl(
|\xi|^2\bigl|\widehat{a}(\xi)\bigr|^2\bigr)^{\frac1{p'}}
\cdot|\xih|^{2(-3\alpha(r)+\theta)}|\xi_3|^{-2\theta}\,d\xi,
\end{align*}
and
similarly due to $\f2{s_1}+\f1{s_2}=\f1{p'}+3\al(r),$ we have
\begin{align*}
\|a\|_{\dH^{\f2{s_1}-6\al(r)+\theta, \f1{s_2}-\th}}^2
&=\int_{\R^3}\bigl|\widehat{a}(\xi)\bigr|^{2}\bigl(|\xih|^{2\left(\f1{s_1}-3\al(r)\right)}
|\xi_3|^{\frac2{s_2}}\bigr)
\cdot|\xih|^{2(-3\alpha(r)+\theta)}|\xi_3|^{-2\theta}\,d\xi\\
&\leq \int_{\R^3}\bigl|\widehat{a}(\xi)\bigr|^{\frac2p}\bigl(
|\xi|^2\bigl|\widehat{a}(\xi)\bigr|^2\bigr)^{\frac1{p'}}
\cdot|\xih|^{2(-3\alpha(r)+\theta)}|\xi_3|^{-2\theta}\,d\xi.
\end{align*}
Applying H\"older's inequality with measure $|\xih|^{2(-3\alpha(r)+\theta)}|\xi_3|^{-2\theta}\,d\xi$ gives
\beno
\begin{split}
\|a\|_{\dH^{\theta,\frac1{p'}-3\al(r)-\theta}}+\|a\|_{\dH^{\f2{s_1}-6\al(r)+\theta, \f1{s_2}-\th}}\lesssim \|a\|_{\htr}^{\frac1p}\|\nabla a\|_{\htr}^{1-\frac1{p}}.
\end{split}
\eeno
As a result, it comes out
\begin{equation}\label{5.23}
\|\p3 v^3\|_{\htrps}
\leq \|\p3 v^3\|_{\htr}^{\frac1p}\|\nabla\p3 v^3\|_{\htr}^{1-\frac1{p}}.
\end{equation}
Substituting \eqref{5.22},~\eqref{5.23} into \eqref{5.21}, and using Young's inequality, we obtain
\begin{equation}\begin{split}\label{5.25}
\bigl|\bigl(Q_1(v,v)|\p3 v^3\bigr)_{\htr}\bigr|
\leq \frac16&\|\nabla\p3 v^3\|_{\htr}^2+C\|v^3\|_{\Bs}^p\|\p3 v^3\|_{\htr}^2\\
&+C\|v^3\|_{\Bs}\bigl\|\omr\bigr\|_{L^2}^{2\bigl(2\alpha(r)+\frac1p\bigr)}
\bigl\|\nabla\omr\bigr\|_{L^2}^{2\bigl(1-\frac1p\bigr)},
\end{split}\end{equation}
with $s_1,s_2$ satisfying \eqref{lem5.1.2}.

\no$\bullet$ \underline{The estimate of $\bigl|\bigl(Q_2(v,v) \, |\, \pa_3v^3\bigr)_{\cH^{\theta,r}}\bigr|$}

We first get, by applying Bony's decomposition, that
\beno
\p3v^\h\cdot\na_\h v^3=\bigl(T^\h+R^\h+\bar{T}^\h\bigr)\bigl(T^\v+R^\v+\bar{T}^\v\bigr)(\p3v^\h, \na_\h v^3).
\eeno
Applying Lemma \ref{lemBern} gives
\beno
\begin{split}
\|S_{k-1}^\h\D_\ell^\v\p3 v^\h\|_{L^\infty_\h(L^2_\v)}\lesssim & \sum_{k'\leq k-1}2^{k'}2^{\ell}\|\D_{k'}^\h\D_{\ell}^\v v^\h\|_{L^2}\\
\lesssim & \sum_{k'\leq k-1} c_{k',\ell}2^{k'\mu}2^{\ell\left(3\al(r)-\mu+\f2p\right)}\|v^\h\|_{\dH^{1-\mu, 1+\mu-3\al(r)-\f2p}}\\
\lesssim & c_{k,\ell}2^{k\mu}2^{\ell\left(3\al(r)-\mu+\f2p\right)}\|v^\h\|_{\dH^{1-\mu, 1+\mu-3\al(r)-\f2p}}.
\end{split}
\eeno
 Using this and Lemma \ref{minusBesov}, we obtain
 \beno
 \begin{split}
 \bigl\|\dhk&\dvl\Th\Tvb(\pa_3 \vh,\na_\h v^3)\bigr\|_{L^2}\\
 \lesssim & \sum_{\substack{|k'-k|\leq 4\\|\ell'-\ell|\leq 4}}
\|S_{k'-1}^\h \D_{\ell'}^\v\p3 v^\h\|_{L^\infty_\h(L^2_\v)}2^{k'}\|\D_{k'}^\h S_{\ell'-1}^\v v^3\|_{L^2_\h(L^\infty_\v)}\\
 \lesssim & c_{k,\ell}2^{k\left(1-\f2p\right)}2^{\ell\left(3\al(r)+\f2p\right)}\|v^3\|_{\bigl(\dB^{\frac2p+\mu}_{2,\infty}\bigr)_{\h}\bigl(\dB^{-\mu}_{\infty,\infty}\bigr)_{\v}}
\|v^\h\|_{\dH^{1-\mu,1+\mu-3\alpha(r)-\frac2p}}
 \end{split}
 \eeno
 While we again deduce from Lemma \ref{minusBesov} and Lemma \ref{lemBern} that
 \beno
 \begin{split}
 \bigl\|\dhk\dvl R^\h\Tvb(&\pa_3 \vh,\na_\h v^3)\bigr\|_{L^2}\lesssim  2^k\sum_{\substack{k'\geq k-3\\|\ell'-\ell|\leq 4}}
 \|\D_{k'}^\h \D_{\ell'}^\v\p3 v^\h\|_{L^2}2^{k'}\|\wt{\D}_{k'}^\h S_{\ell'-1}^\v v^3\|_{L^2_\h(L^\infty_\v)}\\
 \lesssim & 2^k\sum_{\substack{k'\geq k-3\\|\ell'-\ell|\leq 4}}c_{k',\ell'}2^{-k'\th}2^{\ell'(3\al(r)+\th)}\|v^3\|_{\bigl(\dB^{\frac2p+\mu}_{2,\infty}\bigr)_{\h}\bigl(\dB^{-\mu}_{\infty,\infty}\bigr)_{\v}}
\|v^\h\|_{\dH^{1+\th-\mu-\frac2p,1+\mu-3\alpha(r)-\th}}\\
\lesssim &c_{k,\ell}2^{k(1-\th)}2^{\ell(3\al(r)+\th)}\|v^3\|_{\bigl(\dB^{\frac2p+\mu}_{2,\infty}\bigr)_{\h}\bigl(\dB^{-\mu}_{\infty,\infty}\bigr)_{\v}}
\|v^\h\|_{\dH^{1+\th-\mu-\frac2p,1+\mu-3\alpha(r)-\th}}.
\end{split}
 \eeno
 This shows that
 \begin{equation}\begin{split}\label{5.28}
&\bigl\|\Th\Tvb(\pa_3 v^\h,\na_\h v^3)\bigr\|_{\dH^{-1+\frac2{p},-3\alpha(r)-\frac2p}}
+\bigl\|\Rh \Tvb(\pa_3 v^\h,\na_\h v^3)\bigr\|_{\dH^{-1+\th,-3\alpha(r)-\th}}\\
&\qquad\qquad\lesssim\|v^3\|_{\bigl(\dB^{\frac2p+\mu}_{2,\infty}\bigr)_{\h}\bigl(\dB^{-\mu}_{\infty,\infty}\bigr)_{\v}}
\bigl(\|v^\h\|_{\dH^{1-\mu,1+\mu-3\alpha(r)-\frac2p}}
+\|v^\h\|_{\dH^{1+\th-\mu-\frac2p,1+\mu-3\alpha(r)-\th}}\bigr).
\end{split}\end{equation}
While note that
\beno
\bigl\|\dhk\D_\ell^v\Thb\Tvb(\pa_3 v^\h,\na_\h v^3)\bigr\|_{L^2}\lesssim \sum_{\substack{|k'-k|\leq 4\\|\ell'-\ell|\leq 4}}
\|\D_{k'}^\h\D_\ell^\v\p3 v^\h\|_{L^2}\|S_{k'-1}^\h S_{\ell'-1}^\v \na_\h v^3\|_{L^\infty},
\eeno
yet it follows from Lemma \ref{lemBern} that
\beno
\begin{split}
\|S_{k-1}^\h S_{\ell-1}^\v \na_\h v^3\|_{L^\infty}\lesssim &
\sum_{\substack{k'\leq k-2\\ \ell'\leq \ell-2}}2^{k'\left(\f53-\f4{3p}\right)}2^{\ell'\left(\f13-\f2{3p}\right)}\|\D_{k'}^\h\D_{\ell'}^\v v^3\|_{L^{\f{3p}{p-2}}}\\
\lesssim &
2^{k\left(\f53-\f4{3p}\right)}2^{\ell\left(\f13-\f2{3p}\right)}\|\D_{k'}^\h\D_{\ell'}^\v v^3\|_{L^{\f{3p}{p-2}}}.
\end{split}
\eeno
As a result, it comes out
\beno
\bigl\|\dhk\D_\ell^v\Thb\Tvb(\pa_3 v^\h,\na_\h v^3)\bigr\|_{L^2}\lesssim c_{k,\ell}2^{k\left(\f23-\f4{3p}\right)}
2^{\ell\left(3\al(r)+\f13+\f4{3p}\right)}\|v^\h\|_{\dH^{1,1-3\alpha(r)-\frac2p}}
\|v^3\|_{L^{\frac{3p}{p-2}}},
\eeno
and hence
\begin{equation}\begin{split}\label{5.26}
\bigl\|\Thb\Tvb(\pa_3 v^\h,\na_\h v^3)\bigr\|_{\dH^{-\frac23+\frac4{3p},-\frac13-\frac4{3p}-3\alpha(r)}}
\lesssim\|v^\h\|_{\dH^{1,1-3\alpha(r)-\frac2p}}
\|v^3\|_{L^{\frac{3p}{p-2}}}.
\end{split}\end{equation}

Observing that $\p3$ is applied on the low-frequency part in $\Tv(\pa_3 v^\h,\na_\h v^3)$,
but on the high-frequency part in $\Tvb(\pa_3 v^\h,\na_\h v^3)$, hence naturally we believe that
the estimates \eqref{5.28},~\eqref{5.26} still hold when $\Tvb$ is replaced by $\Tv$.
Indeed, exactly along the same line of the proof of these two estimates, we can verify
\begin{equation}\begin{split}
&\bigl\|\Th\Tv(\pa_3 v^\h,\na_\h v^3)\bigr\|_{\dH^{-1+\frac2{p},-3\alpha(r)-\frac2p}}
+\bigl\|\Rh \Tv(\pa_3 v^\h,\na_\h v^3)\bigr\|_{\dH^{-1+\th,-3\alpha(r)-\th}}\\
&\qquad\qquad\lesssim\|v^3\|_{\bigl(\dB^{\frac2p+\mu}_{2,\infty}\bigr)_{\h}\bigl(\dB^{-\mu}_{\infty,\infty}\bigr)_{\v}}
\bigl(\|v^\h\|_{\dH^{1-\mu,1+\mu-3\alpha(r)-\frac2p}}
+\|v^\h\|_{\dH^{1+\th-\mu-\frac2p,1+\mu-3\alpha(r)-\th}}\bigr).
\end{split}\end{equation}
and
\begin{equation}\begin{split}\label{5.27}
\bigl\|\Thb\Tv(\pa_3 v^\h,\na_\h v^3)\bigr\|_{\dH^{-\frac23+\frac4{3p},-\frac13-\frac4{3p}-3\alpha(r)}}
\lesssim\|v^\h\|_{\dH^{1,1-3\alpha(r)-\frac2p}}
\|v^3\|_{L^{\frac{3p}{p-2}}}.
\end{split}\end{equation}

On the other hand, by using Lemma \ref{minusBesov} and interpolation inequality, we have
\begin{align*}
\|S_{k-1}^\h&\D_\ell^\v\na_h v^3\|_{L^2}\lesssim c_{k,\ell}2^{k(\frac13+\frac{4}{3p}+6\al(r)-2\th)}
2^{-\ell(\f43+\f4{3p}+3\al(r)-2\th)}
\|v^3\|_{\dH^{\frac23-\frac{4}{3p}-6\al(r)+2\th,\f43+\f4{3p}+3\al(r)-2\th}}\\
&\lesssim c_{k,\ell}2^{k(\frac13+\frac{4}{3p}+6\al(r)-2\th)}
2^{-\ell(\f43+\f4{3p}+3\al(r)-2\th)}
\|\na_h \pa_3v^3\|_{\htr}^{\frac23-\frac{4}{3p}-3\al(r)+\th}
\|\pa_3^2 v^3\|_{\htr}^{\f13+\f4{3p}+3\al(r)-\th}\\
&\lesssim c_{k,\ell}2^{k(\frac13+\frac{4}{3p}+6\al(r)-2\th)}
2^{-\ell(\f43+\f4{3p}+3\al(r)-2\th)}
\|\na \pa_3v^3\|_{\htr}.
\end{align*}
While it is easy to observe from Lemma \ref{lemBern} that
\beno
\|\D_k^\h\D_\ell^\v \pa_3 v^\h\|_{L^2}\lesssim c_{k,\ell}2^{-k}2^{\ell(\f2{p}+3\al(r))}
\|v^\h\|_{H^{1,1-\frac2p-3\al(r)}}.
\eeno
As a consequence, we obtain
\beno
\begin{split}
\bigl\|\D_k^\h&\D_\ell^\v\Thb R^\v(\p3v^\h,\na_h v^3)\bigr\|_{L^{\f{3p}{2(p+1)}}}\\
\lesssim &2^{\f{\ell}3\left(1-\f2p\right)}
\sum_{\substack{|k'-k|\leq 4\\ \ell'\geq \ell-3}}d_{k',\ell'}2^{k'(6\al(r)-2\th)}
2^{-\ell\left(\f43-\f2{3p}-2\th\right)}\|\na\p3v^3\|_{\htr}\|v^\h\|_{\dH^{1,1-\frac2p-3\alpha(r)}}\\
\lesssim &d_{k,\ell}2^{k(6\al(r)-2\th)}2^{-\ell(1-2\th)}\|\na\p3v^3\|_{\htr}\|v^\h\|_{\dH^{1,1-\frac2p-3\alpha(r)}}.
\end{split}
\eeno
Along the same line, due to $p>4,~r>\frac95,~\th>0$, there holds
\begin{equation}\label{prcondi}
\frac23-\frac4{3p}-6\alpha(r)+2\theta>0.
\end{equation}
Then we have
\beno
\begin{split}
\bigl\|&\D_k^\h\D_\ell^\v R^\h R^\v(\p3v^\h,\na_h v^3)\bigr\|_{L^{\f{3p}{2(p+1)}}}\\
\lesssim &2^{\f{2k}3\left(1-\f2p\right)}2^{\f{\ell}3\left(1-\f2p\right)}\sum_{\substack{k'\geq k-3\\ \ell'\geq \ell-3}}\|\D_{k'}^\h\D_{\ell'}^\v\p3 v^\h\|_{L^2}
\|\wt{\D}_{k'}^\h\wt{\D}_{\ell'}^\v\na_\h v^3\|_{L^2}\\
\lesssim &2^{\f{2k}3\left(1-\f2p\right)}2^{\f{\ell}3\left(1-\f2p\right)}\sum_{\substack{k'\geq k-3\\ \ell'\geq \ell-3}}
d_{k',\ell'}2^{-k'\left(\f23-\f4{3p}-6\al(r)+2\th\right)}
2^{-\ell'\left(\f43-\f2{3p}-2\th\right)}\|\na\p3v^3\|_{\htr}\|v^\h\|_{\dH^{1,1-\frac2p-3\alpha(r)}}\\
\lesssim &d_{k,\ell}2^{k(6\al(r)-2\th)}2^{-\ell(1-2\th)}\|\na\p3v^3\|_{\htr}\|v^\h\|_{\dH^{1,1-\frac2p-3\alpha(r)}}.
\end{split}
\eeno
We thus obtain
\begin{equation}\begin{split}\label{5.29}
\bigl\|(\Thb+\Rh)\Rv(\p3 v^\h,\na_\h v^3)\bigr\|_{\dB^{-6\alpha(r)+2\theta,1-2\theta}_{\frac{3p}{2p+2},1}}
& \lesssim\|\nabla\p3 v^3\|_{\htr}\|v^\h\|_{\dH^{1,1-\frac2p-3\alpha(r)}_{2,1}}.
\end{split}\end{equation}
Similarly, by using interpolation inequality, we have
\beno
\|\D_k^\h\D_\ell^\v \nabla_\h v^3\|_{L^2}\lesssim c_{k,\ell}2^{-k\left(2\th-6\al(r)-\f2p\right)}2^{-\ell\left(1+3\al(r)-2\th+\f2{p}\right)}\|\na\p3v^3\|_{\htr}.
\eeno
Then we deduce from Lemma \ref{minusBesov} and Lemma \ref{lemBern} that
\beno
\begin{split}
\bigl\|\D_k^\h\D_\ell^\v T^\h R^\v&(\p3v^\h,\na_h v^3)\bigr\|_{L^2_\h(L^1_\v)}
\lesssim \sum_{\substack{|k'-k|\leq 4\\ \ell'\geq \ell-3}}\|S_{k'-1}^\h\D_{\ell'}^\v\p3 v^\h\|_{L^\infty_\h(L^2)}
\|\D_{k'}^\h\wt{\D}_{\ell'}^\v\na_\h v^3\|_{L^2}\\
\lesssim &\sum_{\substack{|k'-k|\leq 4\\ \ell'\geq \ell-3}}
d_{k',\ell'}2^{k'\left(\f2{p}+6\al(r)+\mu-2\th\right)}
2^{-\ell'\left(1+\mu-2\th\right)}\|\na\p3v^3\|_{\htr}\|v^\h\|_{\dH^{1-\mu,1-\frac2p-3\alpha(r)+\mu}_{2,1}}\\
\lesssim &d_{k,\ell}2^{k\left(\f2{p}+6\al(r)+\mu-2\th\right)}
2^{-\ell\left(1+\mu-2\th\right)}\|\na\p3v^3\|_{\htr}\|v^\h\|_{\dH^{1-\mu,1-\frac2p-3\alpha(r)+\mu}_{2,1}},
\end{split}
\eeno
which implies
\begin{equation}\label{5.31}
\bigl\|\Th\Rv(\p3 v^\h,\na_\h v^3)\bigr\|
_{\bigl(\dB^{-\frac2p-6\alpha(r)+2\theta-\mu}_{2,1}\bigr)_{\h}\bigl(\dB^{1-2\theta+\mu}_{1,1}\bigr)_{\v}}
\lesssim\|\nabla\p3 v^3\|_{\htr}\|v^\h\|_{\dH^{1-\mu,1-\frac2p-3\alpha(r)+\mu}_{2,1}}.
\end{equation}

Now we are in position to completes the estimate of $\bigl|\bigl(Q_2(v,v) \, |\, \pa_3v^3\bigr)_{\cH^{\theta,r}}\bigr|$. We first get, by using the estimates
\eqref{5.28}-\eqref{5.27}, that
\begin{equation}\begin{split}\label{5.35}
&\Bigl|\bigl((\Id+2\partial_3^2\Delta^{-1}) \Thb(\Tv+\Tvb)(\partial_3 v^\h, \na_\h v^3)
\big|\p3 v^3\bigr)_{\htr}\Bigr|\\
&\lesssim \bigl\|\Thb(\Tv+\Tvb)(\pa_3 v^\h,\na_\h v^3)\bigr\|_{\dH^{-\frac23+\frac4{3p},-\frac13-\frac4{3p}-3\alpha(r)}}
\|\p3 v^3\|_{\dH^{\frac23-\frac4{3p}-6\alpha(r)+2\theta,\frac13+\frac4{3p}+3\alpha(r)-2\theta}}\\
&\lesssim \|v^3\|_{L^{\frac{3p}{p-2}}}\|v^\h\|_{\dH^{1,1-\frac2p-3\alpha(r)}}
\|\nabla\p3 v^3\|_{\htr},
\end{split}\end{equation}
and
\begin{equation}\begin{split}\label{5.36}
\Bigl|\bigl((\Id+&2\partial_3^2\Delta^{-1}) (\Th+\Rh)
(\Tv+\Tvb)(\partial_3 v^\h, \na_\h v^3)
\big|\p3 v^3\bigr)_{\htr}\Bigr|\\
&\lesssim \bigl\|\Th(\Tv+\Tvb)(\pa_3 v^\h,\na_\h v^3)\bigr\|_{\dH^{-1+\frac2p,-\frac2p-3\alpha(r)}}
\|\p3 v^3\|_{\dH^{1-\frac2p-6\alpha(r)+2\theta,\frac2p+3\alpha(r)-2\theta}}\\
&\qquad+\bigl\|\Rh(\Tv+\Tvb)(\pa_3 v^\h,\na_\h v^3)\bigr\|_{\dH^{-1+\th,-\th-3\alpha(r)}}
\|\p3 v^3\|_{\dH^{1-\th-6\alpha(r)+2\theta,3\alpha(r)-\theta}}\\
&\lesssim\|v^3\|_{\bigl(\dB^{\frac2p+\mu}_{2,\infty}\bigr)_{\h}\bigl(\dB^{-\mu}_{\infty,\infty}\bigr)_{\v}}
\|\nabla\p3 v^3\|_{\htr}\\
&\qquad\qquad\qquad\qquad\qquad\times\bigl(\|v^\h\|_{\dB^{1-\mu,1-\frac2p-3\alpha(r)+\mu}_{2,1}}
+\|v^\h\|_{\dB^{1-\frac2p+\th-\mu,1-3\alpha(r)-\th+\mu}_{2,1}}\bigr).
\end{split}\end{equation}
Using \eqref{5.29} and \eqref{5.31}, we obtain
\begin{equation}\begin{split}\label{5.37}
&\Bigl|\bigl((\Id+2\partial_3^2\Delta^{-1}) (\Thb+\Rh)\Rv(\partial_3 v^\h, \na_\h v^3)
\big|\p3 v^3\bigr)_{\htr}\Bigr|\\
&\lesssim \|(\Thb+\Rh)\Rv(\p3 v^\h,\na_\h v^3)\|_{\dB^{-6\alpha(r)+2\theta,1-2\theta}
_{\frac{3p}{2p+2},1}}
\|\p3 v^3\|_{\dB^{0,-1}_{\frac{3p}{p-2},\infty}}\\
&\lesssim\|\nabla\p3 v^3\|_{\htr}\|v^\h\|_{\dB^{1,1-\frac2p-3\alpha(r)}_{2,1}}
\|v^3\|_{L^{\frac{3p}{p-2}}},
\end{split}\end{equation}
and
\begin{equation}\begin{split}\label{5.38}
&\Bigl|\bigl((\Id+2\partial_3^2\Delta^{-1}) \Th\Rv(\partial_3 v^\h, \na_\h v^3)
\big|\p3 v^3\bigr)_{\htr}\Bigr|\\
&\lesssim \|\Th\Rv(\p3 v^\h,\na_\h v^3)\|
_{\bigl(\dB^{-\frac2p-6\alpha(r)+2\theta-\mu}_{2,1}\bigr)_{\h}\bigl(\dB^{1-2\theta+\mu}_{1,1}\bigr)_{\v}}
\|\p3 v^3\|_{\bigl(\dB^{\frac2p+\mu}_{2,\infty}\bigr)_{\h}\bigl(\dB^{-1-\mu}_{\infty,\infty}\bigr)_{\v}}\\
&\lesssim\|\nabla\p3 v^3\|_{\htr}\|\vl\|_{\dB^{1-\mu,1-\frac2p-3\alpha(r)+\mu}_{2,1}}
\|v^3\|_{\bigl(\dB^{\frac2p+\mu}_{2,\infty}\bigr)_{\h}\bigl(\dB^{-\mu}_{\infty,\infty}\bigr)_{\v}}.
\end{split}\end{equation}

In view of Lemma \ref{prop2.20BCD}, we have
$\bigl(\dB^{\frac2p+\frac2{q_1}-1+\mu}_{q_1,\infty}\bigr)_{\h}
\bigl(\dB^{\frac1{q_2}-\mu}_{q_2,\infty}\bigr)_{\v}\hookrightarrow
\bigl(\dB^{\frac2p+\mu}_{2,\infty}\bigr)_{\h}\bigl(\dB^{-\mu}_{\infty,\infty}\bigr)_{\v}$.
Thus we get, by combining \eqref{5.35}-\eqref{5.38} and
using Lemma \ref{BiotSavartBesovanisotropic}, that
\begin{equation}\begin{split}\label{5.40}
&\bigl|\bigl(Q_2(v,v)
\big|\p3 v^3\bigr)_{\htr}\bigr|\\
&\leq C\|v^3\|_{SC}\|\nabla\p3 v^3\|_{\htr}
\Bigl(
\bigl\|\omega_{\frac{r}2}\bigr\|_{L^2} ^{2\alpha(r)+\frac2p} \bigl\|\nabla
\omega_{\frac{r}2}\bigr\|_{L^2} ^{1-\frac2p}
+\|\partial_3v^3\|_{\htr}^{\frac2p} \|\nabla
\partial_3v^3\|_{\htr}^{1-\frac2p}\Bigr)\\
&\leq \frac16\|\nabla\p3 v^3\|_{\htr}^2+C\|v^3\|_{SC}^2
\bigl\|\omega_{\frac{r}2}\bigr\|_{L^2} ^{2\bigl(2\alpha(r)+\frac2p\bigr)} \bigl\|\nabla
\omega_{\frac{r}2}\bigr\|_{L^2} ^{2\bigl(1-\frac2p\bigr)}
+C\|v^3\|_{SC}^p\|\partial_3v^3\|_{\htr}^2.
\end{split}\end{equation}

\no$\bullet$ \underline{The estimate of $\bigl|\bigl(Q_3(v,v) \, |\, \pa_3v^3\bigr)_{\cH^{\theta,r}}\bigr|$}

Let us first deal with the  estimate of  $\bigl|\bigl(\vh\cdot\nablah \p3 v^3 \, |\, \pa_3v^3\bigr)_{\cH^{\theta,r}}\bigr|.$
Applying  Bony's decomposition in the vertical variable for $ \vh\cdot\nablah \p3 v^3$ yields
\beno
\vh\cdot\nablah \p3 v^3=\left(T^\v+R^\v+\Tvb\right)(\vh, \nablah \p3 v^3).
\eeno

We first observe that $\p3$ is applied on the low-frequency part in $\Tvb(\vh,\nablah\p3 v^3)$,
but on the high-frequency part in $\Tvb(\p3\vh,\nablah v^3)$, hence naturally we believe that
the estimates for $\Tvb(\p3\vh,\nablah v^3)$ still hold for $\Tvb(\vh,\nablah\p3 v^3)$.
For the same reason, we also believe that
$\Rv(\vh,\nablah\p3 v^3)$ shares the same estimates as $\Rv(\p3\vh,\nablah v^3)$.
Indeed, exactly along the same line of the estimate for $\bigl|\bigl(Q_2(v,v)
\big|\p3 v^3\bigr)_{\htr}\bigr|$, we achieve
\begin{equation}\begin{split}\label{lem5.3.2}
\bigl|\bigl((\Tvb+\Rv)(\vh,\nablah\p3v^3)
\big|\p3 v^3\bigr)_{\htr}\bigr|
\leq &\frac1{6}\|\nabla\p3 v^3\|_{\htr}^2+C\|v^3\|_{SC}^p\|\partial_3v^3\|_{\htr}^2\\
&+C\|v^3\|_{SC}^2
\bigl\|\omega_{\frac{r}2}\bigr\|_{L^2} ^{2\bigl(2\alpha(r)+\frac2p\bigr)} \bigl\|\nabla
\omega_{\frac{r}2}\bigr\|_{L^2} ^{2\bigl(1-\frac2p\bigr)}.
\end{split}\end{equation}

It remains to deal with the estimate of $T^\v(\vh, \nablah \p3 v^3).$ By using Bony's decomposition in the
horizontal variables for $T^\v(\vh, \nablah \p3 v^3),$ we write
\beno
T^\v(\vh, \nablah \p3 v^3)=\left(T^\h+R^\h+\Thb\right)T^\v(\vh, \nablah \p3 v^3).
\eeno
We first write
\begin{equation}\label{lem5.3.3}
\bigl(\Th\Tv(\vh,\nablah\p3 v^3)
\big|\p3 v^3\bigr)_{\htr}=\sum_{k,\ell\in\Z}2^{2k(-3\alpha(r)+\theta)}2^{-2\ell\theta}
\bigl(I_{k,\ell}^1+I_{k,\ell}^2+I_{k,\ell}^3\bigr),\quad\mbox{with}
\end{equation}
\begin{align*}
I_{k,\ell}^1\eqdefa&\sum_{\substack{|k'-k|\leq 4\\|\ell'-\ell|\leq 4}}\bigl([\dhk\dvl,\Shkp\Svlp\vh]
\dhkp\dvlp\nablah\p3 v^3\big|\dhk\dvl\p3 v^3\bigr)_{L^2},\\
I_{k,\ell}^2\eqdefa&\sum_{\substack{|k'-k|\leq 4\\|\ell'-\ell|\leq 4}}\bigl(
(\Shkp\Svlp\vh-\Shk\Svl\vh)\dhkp\dvlp\dhk\dvl\nablah\p3 v^3\big|\dhk\dvl\p3 v^3\bigr)_{L^2},\\
I_{k,\ell}^3\eqdefa&-\frac12\bigl(\Shk\Svl\divh\vh\cdot\dhk\dvl\p3 v^3\big|\dhk\dvl\p3 v^3\bigr)_{L^2}.
\end{align*}
It follows from a standard commutator's estimate (see for instance \cite{BCD}) that
\begin{equation*}\begin{split}
\bigl|I_{k,\ell}^1\bigr|&\lesssim\sum_{\substack{|k'-k|\leq 4\\|\ell'-\ell|\leq 4}}\Bigl(
2^{-k}\bigl\|\Shkp\Svlp\nablah\vh\bigr\|_{L^{\frac{6p}{p+4}}}\bigl\|\dhkp\dvlp\nablah\p3 v^3\bigr\|_{L^2}
\bigl\|\dhk\dvl\p3 v^3\bigr\|_{L^{\frac{3p}{p-2}}}\\
&\qquad\qquad +2^{-\ell}\bigl\|\Shkp\Svlp\p3\vh\bigr\|_{\Lh^\infty\Lv^2}\bigl\|\dhkp\dvlp\nablah\p3 v^3\bigr\|_{\Lh^2\Lv^\infty}
\bigl\|\dhk\dvl\p3 v^3\bigr\|_{L^{2}}\Bigr)\\
&\eqdefa I_{k,\ell}^{1,1}+I_{k,\ell}^{1,2}.
\end{split}\end{equation*}
Noting that $-3\al(r)+\th<0,~-\th<0$,
we  use Lemmas \ref{minusBesov} and \ref{lemBern} to get
\begin{equation*}\begin{split}
I_{k,\ell}^{1,1}&\lesssim  c_{k,\ell}2^{-k(1-3\alpha(r)+\theta)}2^{\ell\th}
\bigl\|\nablah\vh\bigr\|_{\dH^{\frac23-\frac4{3p}-3\alpha(r)+\theta,\frac13-\frac2{3p}-\theta}}
\\
&\quad\times c_{k,\ell} 2^{k(1+3\alpha(r)-\theta)}
2^{-\ell(1-\theta)}
\bigl\|\p3 v^3\bigr\|_{\dH^{-3\alpha(r)+\theta,1-\theta}}
2^\ell\| v^3\|_{L^{\frac{3p}{p-2}}}\\
&\lesssim d_{k,\ell}2^{k(6\alpha(r)-2\theta)}2^{2\ell\theta}\cdot
\bigl\|\nablah\vh\bigr\|_{\dH^{\frac23-\frac4{3p}-3\alpha(r)+\theta,\frac13-\frac2{3p}-\theta}}
\bigl\|\p3^2v^3\bigr\|_{\htr}
\bigl\|v^3\bigr\|_{L^{\frac{3p}{p-2}}},
\end{split}\end{equation*}
and
\begin{equation*}\begin{split}
I_{k,\ell}^{1,2}&\lesssim c_{k,\ell}2^{k\mu}2^{\ell\bigl(3\alpha(r)+\frac2p-\mu-1\bigr)}
\bigl\|\p3\vh\bigr\|_{\dH^{1-\mu,-3\alpha(r)-\frac2p+\mu}_{2,1}}
2^{k\bigl(1-\mu-\frac2p\bigr)}2^{\ell(1+\mu)}\|v^3\|_{\bigl(\dB^{\frac2p+\mu}_{2,\infty}\bigr)_{\h}
\bigl(\dB^{-\mu}_{\infty,\infty}\bigr)_{\v}}
\\
&\quad\times c_{k,\ell}2^{-k\bigl(1-\frac2p-6\alpha(r)+2\theta \bigr)}2^{-\ell\bigl(\frac2p+3\alpha(r)-2\theta\bigr)}
\bigl\|\p3 v^3\bigr\|_{\dH^{1-\frac2p-6\alpha(r)+2\theta,\frac2p+3\alpha(r)-2\theta}}
\\
&\lesssim d_{k,\ell}2^{2k(3\alpha(r)-\theta)}2^{2\ell\theta}
\|\vh\|_{\dH^{1-\mu,1-3\alpha(r)-\frac2p+\mu}_{2,1}}
\| v^3\|_{\bigl(\dB^{\frac2p+\mu}_{2,\infty}\bigr)_{\h}\bigl(\dB^{-\mu}_{\infty,\infty}\bigr)_{\v}}
\bigl\|\nabla\p3 v^3\bigr\|_{\htr},
\end{split}\end{equation*}
Then by summing up in $k$,~$\ell$, using \eqref{a.1wrt}, Lemmas \ref{BiotSavartBesovanisotropic}
and \ref{BiotSavartomega},
we obtain
\begin{equation}\begin{split}\label{lem5.3.5}
&\sum_{k,\ell\in\Z}2^{2k(-3\alpha(r)+\theta)}2^{-2\ell\theta}\bigl|I_{k,\ell}^{1}\bigr|\\
&\lesssim \|v^3\|_{SC}\bigl\|\nabla\p3 v^3\bigr\|_{\htr}
\cdot\Bigl(\bigl\|(\omega,\pa_3v^3)\bigr\|
_{\dH^{\frac23-\frac4{3p}-3\alpha(r)+\theta,\frac13-\frac2{3p}-\theta}}
+\|\vh\|_{\dB^{1-\mu,1-3\alpha(r)-\frac2p+\mu}_{2,1}}\Bigr)\\
&\lesssim \|v^3\|_{SC}\bigl\|\nabla\p3 v^3\bigr\|_{\htr}
\cdot\Bigl(\bigl\|\omr\bigr\|_{L^2}^{2\alpha(r)+\frac2p}\bigl\|\nabla\omr\bigr\|_{L^2}^{1-\frac2p}
+\bigl\|\p3 v^3\bigr\|_{\htr}^{\frac2p}\bigl\|\nabla\p3 v^3\bigr\|_{\htr}^{1-\frac2p}\Bigr).
\end{split}\end{equation}

The estimate for $I_{k,\ell}^{2}$ is similar to $I_{k,\ell}^{1,2}$, whereas the estimate for
$I_{k,\ell}^{3}$ is similar to $I_{k,\ell}^{1,1}$.
Then we  conclude, by using \eqref{lem5.3.3} and Young's inequality, that
\begin{equation}\begin{split}\label{lem5.3.7}
\bigl|\bigl(\Th\Tv(\vh,\nablah\p3 v^3)
\big|\p3 v^3\bigr)_{\htr}\bigr|\leq & \frac1{36}\|\nabla\p3 v^3\|_{\htr}^2+C\|v^3\|_{SC}^p\|\partial_3v^3\|_{\htr}^2\\
&+C\|v^3\|_{SC}^2
\bigl\|\omega_{\frac{r}2}\bigr\|_{L^2} ^{2\bigl(2\alpha(r)+\frac2p\bigr)} \bigl\|\nabla
\omega_{\frac{r}2}\bigr\|_{L^2} ^{2\bigl(1-\frac2p\bigr)}.
\end{split}\end{equation}

Noting that $\nablah$ is applied on the high-frequency part in $\Th(\vh,\nablah\p3 v^3)$, so exactly along the same line of the proof of \eqref{lem5.3.7},
we  get
\begin{equation}\begin{split}\label{lem5.3.9}
\bigl|\bigl((\Thb+\Rh)\Tv(\vh,\nablah\p3 v^3)
\big|\p3 v^3\bigr)_{\htr}\bigr|
\leq & \frac1{36}\|\nabla\p3 v^3\|_{\htr}^2+C\|v^3\|_{SC}^p\|\partial_3v^3\|_{\htr}^2\\
&+C\|v^3\|_{SC}^2
\bigl\|\omega_{\frac{r}2}\bigr\|_{L^2} ^{2\bigl(2\alpha(r)+\frac2p\bigr)} \bigl\|\nabla
\omega_{\frac{r}2}\bigr\|_{L^2} ^{2\bigl(1-\frac2p\bigr)}.
\end{split}\end{equation}

Combining \eqref{lem5.3.2},~\eqref{lem5.3.7} and \eqref{lem5.3.9} gives rise to
\begin{equation}\begin{split}\label{lem5.3.1}
\bigl|\bigl(\vh\cdot\nablah \p3 v^3
\big|\p3 v^3\bigr)_{\htr}\bigr|\leq &\frac2{9}\|\nabla\p3 v^3\|_{\htr}^2+C\|v^3\|_{SC}^p\|\partial_3v^3\|_{\htr}^2\\
&+C\|v^3\|_{SC}^2
\bigl\|\omega_{\frac{r}2}\bigr\|_{L^2} ^{2\bigl(2\alpha(r)+\frac2p\bigr)} \bigl\|\nabla
\omega_{\frac{r}2}\bigr\|_{L^2} ^{2\bigl(1-\frac2p\bigr)}
.
\end{split}\end{equation}

To estimate $\bigl|\bigl(\v^3\cdot \p3^2 v^3
\big|\p3 v^3\bigr)_{\htr}\bigr|$, we first use integration by parts to get
$$\bigl(\v^3\cdot \p3^2 v^3\big|\p3 v^3\bigr)_{\htr}
=-\frac12 \bigl(\p3\v^3\cdot \p3 v^3\big|\p3 v^3\bigr)_{\htr},$$
then by applying Lemma \ref{lem5.1} and interpolation inequality, we obtain
\begin{equation}\begin{split}\label{lem5.3.13}
\bigl|\bigl(\v^3\cdot \p3^2 v^3\big|\p3 v^3\bigr)_{\htr}\bigr|
&\lesssim\bigl\|\p3 v^3\bigr\|_{\htrps}^2 \|v^3\|_{\Bs}\\
&\lesssim \bigl\|\p3 v^3\bigr\|_{\htr}^{\frac2p}\bigl\|\nabla\p3 v^3\bigr\|_{\htr}^{2\bigl(1-\frac1p\bigr)}\|v^3\|_{\Bs}.
\end{split}\end{equation}
Combining the estimates \eqref{lem5.3.1} and \eqref{lem5.3.13},
and using Young's inequality, we arrive at
\begin{equation}\begin{split}\label{lem5.3.14}
\bigl|\bigl(Q_3(v,v)
\big|\p3 v^3\bigr)_{\htr}\bigr|
\leq \frac13\|&\nabla\p3 v^3\|_{\htr}^2+C\|v^3\|_{SC}^2
\bigl\|\omega_{\frac{r}2}\bigr\|_{L^2} ^{2\bigl(2\alpha(r)+\frac2p\bigr)}
\bigl\|\nabla\omega_{\frac{r}2}\bigr\|_{L^2} ^{2\bigl(1-\frac2p\bigr)}\\
&+C\Bigl(\|v^3\|_{SC}^p+\|v^3\|_{\Bs}^p\Bigr)\|\partial_3v^3\|_{\htr}^2.
\end{split}\end{equation}

Now we are in a position to complete the proof of Proposition \ref{prop2.2}.

\begin{proof}[Proof of Proposition \ref{prop2.2}]
By  the assumptions of Proposition \ref{prop2.2}, we have $q_2<\bigl(\frac1p+3\al(r)+\mu\bigr)^{-1}$, so that  we can choose $s_1,$ and $s_2$ with
$$\frac1{s_2}=\frac1p+3\al(r)+\mu<\frac1{q_2},\quad \frac2{s_1}=1-\frac2p-\mu<\frac2{q_1}.$$
Then in view of Lemma \ref{prop2.20BCD}, there holds
$\bigl(\dB^{\frac2p+\frac1{q_1}-1+\mu}_{q_1,\infty}\bigr)_{\h}
\bigl(\dB^{\frac1{q_2}-\mu}_{q_2,\infty}\bigr)_{\v}\hookrightarrow\Bs$.
Substituting \eqref{5.25},~\eqref{5.40} and \eqref{lem5.3.14}
 into \eqref{5.1} leads to
\begin{align*}
\f{d}{dt}\|\pa_3v^3&(t)\|_{\cH^{\theta,r}}^2+\|\na\pa_3v^3(t)\|_{\cH^{\theta,r}}^2
\leq C\|v^3\|_{SC}^p\|\partial_3v^3\|_{\htr}^2\\
&+C\|v^3\|_{SC}\bigl\|\omega_{\frac{r}2}\bigr\|_{L^2} ^{2\bigl(2\alpha(r)+\frac1p\bigr)}
\bigl\|\nabla\omega_{\frac{r}2}\bigr\|_{L^2} ^{2\bigl(1-\frac1p\bigr)}
+C\|v^3\|_{SC}^2
\bigl\|\omega_{\frac{r}2}\bigr\|_{L^2} ^{2\bigl(2\alpha(r)+\frac2p\bigr)}
\bigl\|\nabla\omega_{\frac{r}2}\bigr\|_{L^2} ^{2\bigl(1-\frac2p\bigr)}.
\end{align*}
Applying Gronwall's inequality  yields \eqref{b.4bqp}. \end{proof}

\renewcommand{\theequation}{\thesection.\arabic{equation}}
\setcounter{equation}{0}

\appendix

\setcounter{equation}{0}
\section{Tool box on Functional spaces}\label{apB}

We first recall  the definition of homogeneous Besov space:
\begin{defi}\label{defbesov}
{\sl  Let $(p,q,r)$ be in~$[1,\infty]^3$ and~$s$ in~$\R$. Let us consider~$u$ in~${\mathcal
S}_h'(\R^d),$ which means that $u$ is in~$\cS'(\R^d)$ and satisfies~$\ds\lim_{j\to-\infty}\|S_ju\|_{L^\infty}=0$. We set
$$
\|u\|_{\dB^s_{p,r}}\eqdefa\big\|\big(2^{js}\|\Delta_j
u\|_{L^{p}}\big)_j\bigr\|_{\ell ^{r}(\ZZ)}.
$$
\begin{itemize}
\item
For $s<\frac{d}{p}$ (or $s=\frac{d}{p}$ if $r=1$), we define $
\dB^s_{p,r}(\R^d)\eqdefa \big\{u\in{\mathcal S}_h'(\R^d)\;\big|\; \|
u\|_{\dB^s_{p,r}}<\infty\big\}.$
\item
If $k\in\N$ and if~$\frac{d}{p}+k\leq
s<\frac{d}{p}+k+1$ (or $s=\frac{d}{p}+k+1$ if $r=1$), then we
define~$ \dB^s_{p,r}(\R^d)$  as the subset of $u$
in~${\mathcal S}_h'(\R^d)$ such that $\partial^\beta u$ belongs to~$
\dB^{s-k}_{p,r}(\R^d)$ whenever $|\beta|=k.$
\end{itemize}
}
\end{defi}
We remark that $\dB^s_{2,2}$
coincides with the classical homogeneous Sobolev spaces $\dH^s$.

When $s<0,$ we also have  the following characterization of Besov spaces
$\dB^s_{p,r}:$

\begin{lem}[Proposition $2.33$ of \cite{BCD}]\label{minusBesov}
{\sl Let $s<0$, $1\leq p,r\leq\infty$ and $u\in{\mathcal S}_h'(\R^d)$.
Then $u$ belongs to $\dB^s_{p,r}(\R^d)$ if and only if
$$\bigl( 2^{js}\|\dot{S}_j u\|_{L^p}\bigr)_{j\in\Z}\in\ell^r.$$
Moreover, for some constant $C$ depending only on the dimension $d$, we have
$$C^{-|s|+1}\|u\|_{\dB^s_{p,r}}\leq\Bigl\|\bigl( 2^{js}\|\dot{S}_j u\|_{L^p}\bigr)_{j\in\Z}\Bigr\|_{\ell^r}
\leq C\bigl(1+\frac{1}{|s|}\bigr)\|u\|_{\dB^s_{p,r}}.$$
}\end{lem}

For the
convenience of the readers, we recall the following anisotropic
Bernstein type lemma from \cite{CZ1, Pa02}:

\begin{lem}
\label{lemBern}
{\sl Let $\cB_{h}$ (resp.~$\cB_{v}$) a ball
of~$\R^2_{h}$ (resp.~$\R_{v}$), and~$\cC_{h}$ (resp.~$\cC_{v}$) a
ring of~$\R^2_{h}$ (resp.~$\R_{v}$); let~$1\leq p_2\leq p_1\leq
\infty$ and ~$1\leq q_2\leq q_1\leq \infty.$ Then there holds:

\smallbreak\noindent If the support of~$\wh a$ is included
in~$2^k\cB_{h}$, then
\[
\|\partial_{x_{\rm h}}^\alpha a\|_{L^{p_1}_{\rm h}(L^{q_1}_{\rm v})}
\lesssim 2^{k\left(|\al|+2\left(1/{p_2}-1/{p_1}\right)\right)}
\|a\|_{L^{p_2}_{\rm h}(L^{q_1}_{\rm v})}.
\]
If the support of~$\wh a$ is included in~$2^\ell\cB_{v}$, then
\[
\|\partial_{x_3}^\beta a\|_{L^{p_1}_{\rm h}(L^{q_1}_{\rm v})}
\lesssim 2^{\ell\left(\beta+(1/{q_2}-1/{q_1})\right)} \|
a\|_{L^{p_1}_{\rm h}(L^{q_2}_{\rm v})}.
\]
If the support of~$\wh a$ is included in~$2^k\cC_{h}$, then
\[
\|a\|_{L^{p_1}_{\rm h}(L^{q_1}_{\rm v})} \lesssim
2^{-kN}\sup_{|\al|=N} \|\partial_{x_{\rm h}}^\al a\|_{L^{p_1}_{\rm
h}(L^{q_1}_{\rm v})}.
\]
If the support of~$\wh a$ is included in~$2^\ell\cC_{v}$, then
\[
\|a\|_{L^{p_1}_{\rm h}(L^{q_1}_{\rm v})} \lesssim 2^{-\ell N}
\|\partial_{x_3}^N a\|_{L^{p_1}_{\rm h}(L^{q_1}_{\rm v})}.
\]
}
\end{lem}

\begin{lem}[Proposition $2.20$ of \cite{BCD}]
\label{prop2.20BCD}
{\sl Let $1\leq p_1\leq p_2\leq\infty,\,$and$\, 1\leq r_1\leq r_2\leq\infty.$
Then for any $1\leq s\leq\infty$, the space $\dB^s_{p_1,r_1}$ is continuously
embedded in $\dB^{s-d(\frac{1}{p_1}-\frac{1}{p_2})}_{p_2,r_2}$.
}\end{lem}

\begin{lem}[Theorem $2.40$ and $2.41$ of \cite{BCD}]
\label{Thm2.40BCD}
{\sl For any $p\in[2,\infty[$, we have
$$\dB^0_{p,2}\hookrightarrow L^p\hookrightarrow\dB^0_{p,p}\quad\mbox{and}\quad
\dB^0_{p',p'}\hookrightarrow L^{p'}\hookrightarrow\dB^0_{p',2}.$$
}\end{lem}

\begin{lem}[Lemma $5.1$ of \cite{CZ5}]
 \label{puisancealphaBesov}
 {\sl Let~$(s,\al)$ be in~$]0,1[^2$ and~$(p,q)$ in~$[1,\infty]^2$. For any function $G$ from $\R$ to $\R$ which is H\"olderian
of exponent~$\al$, and any $a\in\dB^s_{p,q},$ one has
 $$
 \|G(a)\|_{\dB^{\al s}_{\frac p\al,\frac q\al} }\lesssim \|G\|_{C^\al} \bigl(\|a\|_{\dB^s_{p,q}}\bigr)^\al\with
\|G\|_{C^\al} \eqdefa \sup_{r\not =r' } \frac
{|G(r)-G(r')|}{|r-r'|^\al} \,\cdotp
 $$
 }
 \end{lem}

\begin{lem}[Lemma $4.3$ of \cite{CZ5}]
\label{lem4.3ofCZ5}
 {\sl For any~$s$ positive and any~$\theta$ in~$]0,s[$,
we have
$$
\|f\|_{(\dB^{s-\theta}_{p,q})_{\rm h}(\dB^{\theta}_{p,1})_{\rm v}}
\lesssim \|f\|_{\dB^{s}_{p,q}}.
$$
}\end{lem}

We shall frequently use the following non-linear interpolation inequalities.
\begin{lem}[Lemma $3.1$ of \cite{CZZ}]
\label{BiotSavartomega}
{\sl For $r$ in~ $]3/2,2[,$ we have
\beq
\label{estimbasomega34}
 \|\nabla a\|_{L^{r}} \lesssim \bigl\|\nabla a_{\frac{r}2}\bigr\|_{L^2}
 \bigl\|a_{\frac{r}2}  \bigr\|_{L^2}^{\frac 2 {r}   -1}.
 \eeq
Moreover, for~$s$  in $\ds[-3\al(r)\,\virgp\, 1-\al(r)]$, we have
\beq
\label{estimbasomega34.0}
\|a\|_{\dH^s} \leq C
\|a_{\frac{r}2}\|_{L^2} ^{1-\al(r)-s}  \|\nabla
a_{\frac{r}2}\|_{L^2} ^{3\al(r)+s} .
 \eeq }
\end{lem}

The following lemma is one of the main motivations of using anisotropic Besov space.
It can be viewed as a generalization of Proposition $2.1$ in \cite{LZ3} and Proposition $3.1$
in \cite{CZZ}, and its proof follows immediately by combining the proofs of these two propositions together.
\begin{lem}
\label{BiotSavartBesovanisotropic}
{\sl Let $\theta\in]0,\al(r)[$ and $s<\frac23\alpha(r),
~\beta<\min\bigl\{1-\frac83\alpha(r),\, 1-3\alpha(r)+\theta\bigr\}$ satisfy $s+\beta>0$.
Let~$v$ be a divergence free vector
field and $\omega=\pa_1 v^2-\pa_2 v^1$.
Then one has
$$
\|v^{\rm h}\|_{\bigl(\dB^{1-s}_{2,1}\bigr)_{\rm
h}\bigl(\dB^{1-3\al(r)-\beta}_{2,1}\bigr)_{\rm v}}\lesssim \bigl\|\,
\omega_{\frac{r}2}\bigr\|_{L^2} ^{2\alpha(r)+s+\beta} \bigl\|\nabla
\omega_{\frac{r}2}\bigr\|_{L^2} ^{1-s-\beta}
+\|\partial_3v^3\|_{\htr}^{s+\beta} \|\nabla
\partial_3v^3\|_{\htr}^{1-s-\beta}.
$$
}
\end{lem}

At the end of this section, let us recall the para-differential
decomposition (Bony's decomposition) from \cite{Bo81}: let $a$ and~$b$ be in $\cS'(\R^3)$,
then we have the following decomposition
\begin{equation}\label{pd}\begin{split}
 &\qquad\qquad\qquad ab=T(a,b)+\bar{T}(a,b)+ R(a,b)\quad\mbox{with}\\
& T(a,b)=\sum_{j\in\Z}S_{j-1}a\Delta_jb, \quad
\bar{T}(a,b)=T(b,a), \quad R(a,b)=\sum_{j\in\Z}\Delta_ja\widetilde{\Delta}_{j}b.
\end{split}\end{equation}
We shall also use Bony's decomposition in horizontal variables or
vertical variable, in order to study product laws
between distributions in anisotropic Besov spaces.

\setcounter{equation}{0}
\section{The proofs of the estimates \eqref{S4eq1}-\eqref{S4eq6}}\label{apA}

\begin{lem}\label{S4lem0}
{\sl Let $q_3>r',$ $\al(r)<\min\bigl\{\f1p,\f1{q_3}\bigr\}$ and $\d\in \bigl]0, \f1{q_3}-\al(r)\bigr[.$ Then we have
 \begin{align*}
 &\|S_{k}^\h \D_{\ell}^\v\omro\|_{L^{\infty}_\h(L^{r'}_\v)}\lesssim c_{(k,\ell),r'}2^{k\left(\d+\f2p-2\al(r)\right)}2^{-\ell\d}\bigl\|\omega_{\frac{r}2}
 \bigr\|_{\dot H^\s}^{\f2{r'}},\\
 &\|S_k^\h S_\ell^\v\omro\|_{L^\infty_\h(L^{m_1}_\v)}\lesssim c_{(k,\ell),r'}2^{k\left(\d+\f2p-2\al(r)\right)}2^{\ell\left(\f1{q_3}-\d\right)}\bigl\|\omega_{\frac{r}2}\bigr\|_{H^\s}^{\f2{r'}}\
 \mbox{with}\
 \f1{m_1}=\f1{r'}-\f1{q_3},\\
 &\|S_k^\h S_\ell^\v\omro\|_{L^\infty_\h(L^{m_2}_\v)}\lesssim c_{(k,\ell),r'}2^{k\left(\d+\f2p-2\al(r)\right)}2^{\ell\left(\f1{q_3}-\al(r)-\d\right)}\bigl\|\omega_{\frac{r}2}\bigr\|_{H^\s}^{\f2{r'}}\
 \mbox{with}\ \f1{m_2}=\f1{2}-\f1{q_3}.
 \end{align*}
 Here and in all that follows, we always denote $\left( c_{(k,\ell),r} \right)_{k,\ell\in\Z^2}$ to be a generic element
 in $\ell^r(\Z^2)$ so that $\sum_{(k,\ell)\in\Z^2}c_{(k,\ell),r}^r=1.$ In particular, when $r=2,$
 $\left( c_{(k,\ell),2} \right)_{k,\ell\in\Z^2}$ is the same to 
 the $\left( c_{k,\ell} \right)_{k,\ell\in\Z^2}$} defined before.
 \end{lem}

 \begin{proof} Note that $\al(r)<\f1p$ and $\d>0,$ we get, by applying Lemma \ref{lemBern} and Lemma \ref{lem4.3ofCZ5}, that
 \beno
 \begin{split}
 \|S_{k}^\h \D_{\ell}^\v\omro\|_{L^{\infty}_\h(L^{r'}_\v)}\lesssim &\sum_{k'\leq k-1}2^{\f{2k'}{r'}}\|\D_{k'}^\h\D_{\ell'}^\v \omro\|_{L^{r'}}\\
 \lesssim &\sum_{k'\leq k-1}c_{(k',\ell'),r'}2^{k'\left(\d+\f2p-2\al(r)\right)}2^{-\ell'\d}\|\omro\|_{B_{r',r'}^{1-\f2p-\d,\d}}\\
 \lesssim &c_{(k,\ell),r'}2^{k\left(\d+\f2p-2\al(r)\right)}2^{-\ell\d}\bigl\|\omega_{\frac{r}2}\bigr\|_{H^\s}^{\f2{r'}}.
 \end{split}
 \eeno This proves the first inequality of the lemma.
 Applying Lemma \ref{lemBern} once again yields
 \beno
 \begin{split}
 \|S_k^\h S_\ell^\v\omro\|_{L^\infty_\h(L^{m_1}_\v)}\lesssim & \sum_{\ell'\leq\ell-1}2^{\ell'\left(\f1{r'}-\f1{m_1}\right)}
\|S_{k}^\h \D_{\ell'}^\v\omro\|_{L^{\infty}_\h(L^{r'}_\v)} \\
 \lesssim &\sum_{k'\leq k-1}\sum_{\ell'\leq\ell-1}c_{(k',\ell'),r'} 2^{k'\left(\d+\f2p-2\al(r)\right)}2^{\ell'\left(\f1{q_2}-\d\right)}\bigl\|\omega_{\frac{r}2}\bigr\|_{H^\s}^{\f2{r'}}\\
 \lesssim & c_{(k,\ell),r'}2^{k\left(\d+\f2p-2\al(r)\right)}2^{\ell\left(\f1{q_3}-\d\right)}\bigl\|\omega_{\frac{r}2}\bigr\|_{H^\s}^{\f2{r'}}.
 \end{split}
 \eeno
 This proves the second inequality of the lemma. The remaining one can be proved along the same line.
\end{proof}

\begin{proof}[Proof of  \eqref{S4eq1}]  We first get, by applying Bony's decomposition in the vertical variable, that
\beno
\label{S4eq2}
\Th(\omro,a)=\Th(\Tv+\bar{T}^\v+R^\v)(\omro,a).
\eeno
By using Lemma \ref{S4lem0} and Lemma \ref{lemBern}, we obtain
\beno
\begin{split}
\|\D_k^\h\D_\ell^\v\Th\Tv(\omro,a)& \|_{L^{r'}}
\lesssim\sum_{|k'-k|\leq 4}\sum_{|\ell'-\ell|\leq 4}\|S_{k'-1}^\h S_{\ell'-1}^\v\omro\|_{L^{\infty}_\h(L^{m_1}_\v)}
\|\D_{k'}^\h\D_{\ell'}^\v a\|_{L^{r'}_\h(L^{q_2}_\v)}\\
&\lesssim\sum_{|k'-k|\leq 4}\sum_{|\ell'-\ell|\leq 4}
2^{2k'\al(r)}2^{-k'\left(\f2p+\mu\right)} 2^{-\ell'\left(\f1{q_2}-\mu\right)}
\|a\|_{\bigl(\dB_{2,\frac{2r'}{r'-2}}^{\f2p+\mu}\bigr)_\h
\bigl(B_{q_2,\frac{2r'}{r'-2}}^{\f1{q_2}-\mu}\bigr)_{\rm v}}\\
&\qquad\qquad\times 2^{k'\left(\d+\f2p-2\al(r)\right)}2^{\ell'\left(\f1{q_2}-\d\right)}
\bigl\|\omega_{\frac{r}2}\bigr\|_{H^\s}^{\f2{r'}}
\cdot c_{k',\ell'},
\end{split}
\eeno
where $m_1$ satisfies $\f1{m_1}=\f1{r'}-\f1q.$
Due to $q_2\in \bigl[2, \bigl(\mu+3\al(r)+1/p\bigr)^{-1}\bigr[,$ one has $\mu<\f1{q_2}-\al(r).$ So that
taking $\d=\mu$ in the above inequality leads to
\beno
\|\D_k^\h\D_\ell^\v\Th\Tv(\omro,a)\|_{L^{r'}}\lesssim c_{k,\ell}\|a\|_{\cB^{\mu,p}_{2,q_2,r}}\bigl\|\omega_{\frac{r}2}\bigr\|_{H^\s}^{\f2{r'}},
\eeno

Similarly, we have
\beno
\|\D_k^\h\D_\ell^\v\Th\bar{T}^\v(\omro,a)\|_{L^{r'}}\lesssim 
\sum_{|k'-k|\leq 4}\sum_{|\ell'-\ell|\leq 4}
\|S_{k'-1}^\h \D_{\ell'}^\v\omro\|_{L^{\infty}_\h(L^{r'}_\v)}
\|\D_{k'}^\h S_{\ell'-1}^\v a\|_{L^{r'}_\h(L^\infty_\v)},
\eeno
from which, Lemma \ref{S4lem0}, we deduce that $\Th\bar{T}^\v(\omro,a)$ shares the same estimate as $\Th\Tv(\omro,a).$

Whereas applying Lemma \ref{lemBern} and then Lemma \ref{S4lem0} yields
\beno
\begin{split}
\|\D_k^\h\D_\ell^\v\Th R^\v(\omro,a)\|_{L^{r'}}\lesssim & 2^{\f{\ell}{q_2}}
\sum_{|k'-k|\leq 4}\sum_{\ell'\geq \ell-3}
\|S_{k'-1}^\h \D_{\ell'}^\v\omro\|_{L^{\infty}_\h(L^{r'}_\v)}
\|\D_{k'}^\h \wt{\D}_{\ell'}^\v a\|_{L^{r'}_\h(L^{q_2}_\v)}\\
\lesssim & 2^{\f{\ell}{q_2}}\sum_{|k'-k|\leq 4}\sum_{\ell'\geq \ell-3}c_{k',\ell'} 2^{-\f{\ell'}q}\|a\|_{\cB^{\mu,p}_{2,q_2,r}}\bigl\|\omega_{\frac{r}2}\bigr\|_{H^\s}^{\f2{r'}}\\
\lesssim& c_{k,\ell}\|a\|_{\cB^{\mu,p}_{2,q_2,r}}\bigl\|\omega_{\frac{r}2}\bigr\|_{H^\s}^{\f2{r'}}.
\end{split}
\eeno
As a result, it comes out
\beq \label{S4eq3}
\|\Th(\omro,a)\|_{L^{r'}}\lesssim\|\Th(\omro,a)\|_{\dB^{0}_{r',2}}\lesssim \|a\|_{\cB^{\mu,p}_{2,q_2,r}}\bigl\|\omega_{\frac{r}2}\bigr\|_{H^\s}^{\f2{r'}}.
\eeq

Along the same line to proof of \eqref{S4eq3}, we deduce from Lemmas \ref{lemBern} and  \ref{S4lem0} that
\beno
\begin{split}
\|\D_k^\h\D_\ell^\v\Th\Tv(\omro,&a)\|_{L^{2}}\lesssim
\sum_{|k'-k|\leq 4}\sum_{|\ell'-\ell|\leq 4}
\|S_{k'-1}^\h S_{\ell'-1}^\v\omro\|_{L^{\infty}_\h(L^{m_2}_\v)}
\|\D_{k'}^\h\D_{\ell'}^\v a\|_{L^{2}_\h(L^{q_2}_\v)}\\
\lesssim &\sum_{|k'-k|\leq 4}\sum_{|\ell'-\ell|\leq 4}c_{k',\ell'}
2^{-k'\left(2\al(r)+\mu-\d\right)}2^{-\ell'\left(\al(r)-\mu+\d\right)}
\|a\|_{\cB^{\mu,p}_{2,q_2,r}}\bigl\|\omega_{\frac{r}2}\bigr\|_{H^\s}^{\f2{r'}},
\end{split}
\eeno
where $m_2$ satisfies $\f1{m_2}=\f1{2}-\f1{q_2}.$ Due to $\mu\in\bigl]\al(r),\frac1{q_2}-3\al(r)-\frac1p\bigr[$
and $\th\in]0,\al(r)[$, we have $\mu-\al(r)+\th\in \bigl]0,\f1{q_2}-\al(r)\bigr[.$
T hus we can take $\d=\mu-\al(r)+\th$ in the above inequality to get
\beno
\|\D_k^\h\D_\ell^\v\Th\Tv(\omro,a)\|_{L^{2}}\lesssim c_{k,\ell}2^{-k(3\al(r)-\th)}2^{-\ell\th}\|a\|_{\cB^{\mu,p}_{2,q_2,r}}\bigl\|\omega_{\frac{r}2}\bigr\|_{H^\s}^{\f2{r'}}.
\eeno

Similarly, for $\r_1=1/\al(r),$  we write
\beno
\|\D_k^\h\D_\ell^\v\Th\bar{T}^\v(\omro,a)\|_{L^{2}}
\lesssim \sum_{|k'-k|\leq 4}\sum_{|\ell'-\ell|\leq 4}
\|S_{k'-1}^\h \D_{\ell'}^\v\omro\|_{L^{\infty}_\h(L^{r'}_\v)}
\|\D_{k'}^\h S_{\ell'-1}^\v a\|_{L^{2}_\h(L^{\r_1}_\v)},
\eeno
and it is easy to observe from Lemma \ref{minusBesov},~Lemma \ref{lemBern} and $\mu>\al(r)$ that
\beno
\|\D_{k'}^\h S_{\ell'-1}^\v a\|_{L^{2}_\h(L^{\r_1}_\v)}\lesssim c_{(k',\ell'),\f{2r}{2-r}}2^{-k'\left(\mu+\f2p\right)}2^{\ell'\left(\mu-\al(r)\right)}\|a\|_{\cB^{\mu,p}_{2,q_2,r}},
\eeno
from which and Lemma \ref{S4lem0} with  $\d=\mu-\al(r)+\th$, we infer
\beno
\begin{split}
\|\D_k^\h\D_\ell^\v\Th\bar{T}^\v(\omro,a)\|_{L^{2}} \lesssim c_{k,\ell}2^{-k(3\al(r)-\th)}2^{-\ell\th}\|a\|_{\cB^{\mu,p}_{2,q_2,r}}\bigl\|\omega_{\frac{r}2}\bigr\|_{H^\s}^{\f2{r'}}.
\end{split}
\eeno
Finally applying Lemma \ref{lemBern} and then Lemma \ref{S4lem0} with  $\d=\mu-\al(r)+\th$ yields
\beno
\begin{split}
\|\D_k^\h\D_\ell^\v\Th R^\v(\omro,a)&\|_{L^{2}}\lesssim  2^{\ell\left(\f{1}{q_2}-\al(r)\right)}\sum_{\substack{|k'-k|\leq 4\\\ell'\geq \ell-3}}
\|S_{k'-1}^\h \D_{\ell'}^\v\omro\|_{L^{\infty}_\h(L^{r'}_\v)}
\|\D_{k'}^\h \wt{\D}_{\ell'}^\v a\|_{L^{2}_\h(L^{q_2}_\v)}\\
\lesssim & 2^{\ell \left(\f{1}{q_2}-\al(r)\right)}\sum_{\substack{|k'-k|\leq 4\\\ell'\geq \ell-3}}c_{k',\ell'}
2^{-k'\left(3\al(r)-\th\right)}2^{-\ell'\left(\th-\al(r)+\f1{q_2}\right)}
\|a\|_{\cB^{\mu,p}_{2,q_2,r}}\bigl\|\omega_{\f{r}2}\bigr\|_{H^\s}^{\f2{r'}}\\
\lesssim& c_{k,\ell}2^{k\left(3\al(r)-\th\right)}2^{-\ell\th}\|a\|_{\cB^{\mu,p}_{2,q_2,r}}\bigl\|\omega_{\frac{r}2}\bigr\|_{H^\s}^{\f2{r'}}.
\end{split}
\eeno
Hence, we achieve
\begin{equation*}\label{phthomro'}
\|\Th(\omro,a)\|_{\dH^{3\alpha(r)-\theta,\theta}}\lesssim \|a\|_{\cB^{\mu,p}_{2,q_2,r}}\bigl\|\omega_{\frac{r}2}\bigr\|_{H^\s}^{\f2{r'}}.
\end{equation*}
Together with \eqref{S4eq3}, we obtain \eqref{S4eq1}.
\end{proof}

\begin{proof}[Proof of  \eqref{S4eq5}] We first get, by applying Bony's decomposition
in the vertical variable, that
\beno
\begin{split}
&\Th({\partial_{\rm h} a},  \omega_{r-1})=\Th(\Tv+\bar{T}^\v+R^\v)({\partial_{\rm h} a},  \omega_{r-1}) \andf \\
 &\Th({\partial_{\rm h}\omega_{r-1} }, a)=\Th(\Tv+\bar{T}^\v+R^\v)({\partial_{\rm h}\omega_{r-1} }, a).
 \end{split}
\eeno
Since the estimates of the above terms are similar, we only present the estimates to the typical terms above.
Noting $\d_1\in \bigl]\mu-\al(r),1-2/p\bigr[,$ we can use Lemma \ref{S4lem0} to obtain
\begin{align*}
\bigl\|\D_k^\h& \D_\ell^\v\Th R^\v({\partial_{\rm h} a},  \omega_{r-1}) \bigr\|_{L^{\f{2r'}{2+r'}}_\h(L^2_\v)}
\lesssim\sum_{|k'-k|\leq 4}\sum_{\ell'\geq \ell-3}
\|S_{k'-1}^\h\D_{\ell'}^v\pa_\h a\|_{L^2_\h(L^q_\v)}\|\D_{k'}^\h\wt{\D}_{\ell'}^\v\omro\|_{L^{r'}_\h(L^{m_2}_\v)}\\
& \lesssim \sum_{|k'-k|\leq 4}\sum_{\ell'\geq \ell-3}
c_{k',\ell'}2^{-k'(\mu-\d_1)}2^{-\ell'(\d_1+\al(r)-\mu)}
\|\pa_\h a\|_{\bigl(\dB_{2,\frac{2r'}{r'-2}}^{-1+\f2p+\mu}\bigr)_\h
\bigl(B_{q_2,\frac{2r'}{r'-2}}^{\f1{q_2}-\mu}\bigr)_{\rm v}}\|\omro\|_{B^{1-\f2p-\d_1,\d_1}_{r',r'}}\\
& \lesssim c_{k,\ell}2^{-k(\mu-\d_1)}2^{-\ell(\d_1+\al(r)-\mu)}
\|a\|_{\cB_{2,q_2,r}^{\mu,p}}\bigl\|\omega_{\f{r}2}\bigr\|_{H^\s}^{\f2{r'}}.
\end{align*}

Along the same line, for $\rho_1=\f1{\al(r)},$ we infer
\beno
\begin{split}
\bigl\|\D_k^\h\D_\ell^\v\Th R^\v(\omega_{r-1}& ,{\partial_{\rm h} a})\bigr\|_{L^{\f{2r'}{2+r'}}_\h(L^2_\v)}
\lesssim \sum_{|k'-k|\leq 4}\sum_{\ell'\geq\ell-3}
\|S_{k'-1}^\h \D_{\ell'}^\v\pa_\h\omro\|_{L^{r'}}\|\D_{k'}^\h\D_{\ell'}^\v a\|_{L^2_\h(L^{\r_1}_\v)}\\
\lesssim &\sum_{|k'-k|\leq 4}\sum_{\ell'\geq\ell-3}
c_{k',\ell'}2^{-k'(\mu-\d_1)}2^{-\ell'(\d_1+\al(r)-\mu)}
\|a\|_{\cB^{\mu,p}_{2,q_2,r}}\|\omro\|_{B^{1-\f2p-\d_1,\d_1}_{r',r'}}\\
\lesssim & c_{k,\ell}2^{-k(\mu-\d_1)}2^{-\ell(\d_1+\al(r)-\mu)}
\|a\|_{\cB^{\mu,p}_{2,q_2,r}}\bigl\|\omega_{\f{r}2}\bigr\|_{H^\s}^{\f2{r'}}.
\end{split}
\eeno
The remaining terms can be handled along the same line.
\end{proof}

\begin{proof}[Proof of \eqref{S4eq6}] Applying Bony's decomposition in the vertical variable gives
 \beno
 R^\h({\partial_{\rm h} a},\omro)=\Rh(\Tv+\bar{T}^\v+R^\v)({\partial_{\rm h} a},\omro).
 \eeno
 We just present the estimate to the typical last term. Indeed, we have
 \beno
 \begin{split}
 \bigl\|\D_k^\h\D_\ell^\v\Rh R^\v(&{\partial_{\rm h} a},\omro)\bigr\|_{L^{\f{q_1r'}{q_1+r'}}_\h(L^2_\v)}
 \lesssim \sum_{k'\geq k-3}\sum_{\ell'\geq\ell-3}\|\D_{k'}^\h\D_{\ell'}^\v\pa_\h a\|_{L^{q_1}_\h(L^{q_2}_\v)}\|\wt{\D}_{k'}^\h\wt{\D}_{\ell'}^\v\omro\|_{L^{r'}_\h(L^{m_2}_\v)}\\
& \lesssim\sum_{k'\geq k-3}\sum_{\ell'\geq\ell-3} c_{k',\ell'}2^{-k'\left(\mu+\f2{q_1}-\d_2-1\right)}2^{-\ell'(\al(r)+\d_2-\mu)}
 \|a\|_{\cB_{q_1,q_2,r}^{\mu,p}}\|\omro\|_{B^{1-\f2p-\d_2,\d_2}_{r',r'}}\\
&\lesssim c_{k,\ell}2^{-k\left(\mu+\f2{q_1}-\d_2-1\right)}2^{-\ell(\al(r)+\d_2-\mu)}
 \|a\|_{\cB_{q_1,q_2,r}^{\mu,p}}\bigl\|\omega_{\f{r}2}\bigr\|_{H^\s}^{\f2{r'}}.
 \end{split}
 \eeno
 The remaining terms can be handled along the same line.
 \end{proof}

\bigbreak \noindent {\bf Acknowledgments.}
P. Zhang is partially supported
by NSF of China under Grants   11731007 and 11688101, Morningside Center of Mathematics of The Chinese Academy of Sciences and innovation grant from National Center for
Mathematics and Interdisciplinary Sciences.

\end{document}